\documentclass{siamart1116}

\usepackage{lipsum}
\usepackage{amsfonts}
\usepackage{graphicx}
\usepackage{epstopdf}
\usepackage{algorithmic}
\ifpdf
  \DeclareGraphicsExtensions{.eps,.pdf,.png,.jpg}
\else
  \DeclareGraphicsExtensions{.eps}
\fi

\numberwithin{theorem}{section}

\newcommand{\TheTitle}{Subdifferential formulae for the supremum of an arbitrary family of functions} 
\newcommand{\TheAuthors}{P\'erez-Aros, P.}

\headers{Subdifferential of the supremum function}{\TheAuthors}

\title{{\TheTitle}\thanks{Submitted to the editors DATE.
\funding{CONICYT-PCHA/doctorado nacional/ 2014-21140621 }}}

\author{
  Pedro P\'erez-Aros\thanks{Instituto de Ciencias de la Ingeniería, Universidad de O'Higgins, Chile
    (\email{pedro.perez@uoh.cl}})
}

\usepackage{amsopn}

\ifpdf
\hypersetup{
  pdftitle={\TheTitle},
  pdfauthor={\TheAuthors}
}
\fi
\newtheorem{example}[theorem]{Example}
\newtheorem{remark}[theorem]{Remark}
\usepackage{enumitem}
\usepackage{cite}
\let\epsilon\varepsilon
\DeclareMathOperator{\ri}{ri}
\newcommand{\N}{\mathbb{N}}
\newcommand{\Rex}{\overline{\mathbb{R}}}
\newcommand{\R}{\mathbb{R}}
\DeclareMathOperator{\supp}{supp}

\DeclareMathOperator{\sub}{\partial}
\DeclareMathOperator{\cl}{cl}

\DeclareMathOperator{\co}{co}
\DeclareMathOperator{\cco}{\overline{co}}
\DeclareMathOperator{\aff}{aff}

\DeclareMathOperator{\epi}{epi}
\DeclareMathOperator{\dom}{dom}
\DeclareMathOperator{\inte}{int}
\DeclareMathOperator{\Pf}{\mathcal{P}_{\textnormal{\tt f}}}
\begin{document}

\maketitle

\begin{abstract}
  	 This work provides calculus for the Fr\'echet and limiting subdifferential of the pointwise supremum given by an arbitrary family of lower semicontinuous functions. We start our study showing   fuzzy results about the  Fr\'echet subdifferential of the supremum function. Posteriorly, we study  in finite- and infinite-dimensional settings the  limiting subdifferential of the supremum function. Finally, we apply our results to the study of the convex subdifferential; here  we recover general formulae for the subdifferential of an arbitrary family of convex functions.
\end{abstract}



\section{Introduction}

Many mathematical models  concern the study of a constraint minimization problem represented by
\begin{align}\label{PROBLEM1}
\begin{array}{c}
\text{minimize }g \;\;\;\text{ subject to}\\
f_t (x) \leq 0, \; \text{ for all } t\in T\text{ and }x\in X,
\end{array}
\end{align} 
where $T$ is an index set and the function $g$ and $f_t$ are defined in some space $X$. In these applications the (possibly nonsmooth) pointwise supremum $f: =\sup f_t$  plays a crucial role in solving this optimization problem, because the constraint $f_t(x) \leq 0$ for all $t\in T$ can be recast as  one single inequality constraint passing to the supremum function $f:=\sup_{T} f_t$. For that reason,  understanding the subdifferential of the function $f$ is decisive in computing necessary optimality conditions. Problem \cref{PROBLEM1} has been widely studied when the index set $T$ is finite,  and nowadays these results are available in numerous monographs of optimization and variational analysis  (see for instance \cite{MR2986672,MR2191744,mordukhovich2018variational,MR2191745,MR1491362,MR2144010,MR1058436,Clarke:1998:NAC:274798}). 

When the set $T$ is infinite \cref{PROBLEM1} is understood to be  a problem of \emph{infinite programming}, and when the space $X$ is finite-dimensional the more precise terminology of \emph{semi-infinite programming} appears due to the  finite-dimensionality  of the variable $x\in X$ and the infinitude of  $T$. These classes of problems  have been  studied over  the last sixty years  by many researchers for the reason that several models in science can be represented as a constraint of the state or the control of a system during a period of time or in a region of the space. Within this framework,  a classical  assumption is the compactness of the set $T$ together with  some hypothesis about the continuity of the function $(t,x) \to f_t(x)$  and its gradient; in this context the set of active indices  $T(x):=\{ t \in T : f_t(x)=f(x)  \}$ performs an important part in the study (see, e.g., \cite{MR2295358}). 

More recent papers have studied the convex subdifferential of the supremum function when $T$ is an arbitrary index set and  $\{ f_t : t\in T \}$ is an arbitrary  family of (possibly non-smooth)  convex  functions (see, for example, \cite{MR2489616,MR3561780,MR2448918,MR2837551,MR2286967,Perez-Aros2018} and the reference therein). Due to the possible emptiness of the set of active indices at a given point $x$, the authors have considered the $\epsilon$-active index set $T_\epsilon(x):=\{  t \in T : f_t(x)\geq f(x) -\epsilon  \}$.  In these works   researchers have successfully calculated  the convex subdifferential of the supremum function without any qualification about the data functions $f_t's$, using the set of $\epsilon$-active indices, the \emph{$\epsilon$-subdifferential} of the data and the \emph{normal cone} of the domain of the function $f$, all of which are well-known concepts in convex analysis.  

When the data functions $\{ f_t\}_{t\in T}$  are non-convex and non-smooth, but \emph{uniformly locally Lipschitz at point $\bar{x}$}, which means, there are  constants $k, \epsilon > 0$ such that
\begin{align}\label{uniformlocally}
| f_t(x) -f_t(y)| \leq k \| x -y\|, \forall x \in \mathbb{B}(\bar{x},\epsilon),\;\forall t \in T,
\end{align} 
we can refer to the classical result about the upper-estimate of the Clarke subdifferential of the function $f$ at the point $\bar{x}$ (see \cite[Theorem 2.8.2]{MR1058436}). It is important to recall that in this result the   set $T$ is compact  and the function $t\to f_t(x)$ is upper-semi continuous for each $x \in \mathbb{B}(\bar{x},\epsilon)$. Recently, in \cite{MR3033113} (see also \cite{MR3205549}) the authors  studied the limiting subdifferential of the function $f$ at $\bar{x}$; they assumed that $T$ is an arbitrary index and the functions $\{f_t\}_{t\in T}$ satisfy \cref{uniformlocally}. They  provided new upper-estimates and improvements of  the mentioned result relative to the  Clarke subdifferential. Using these calculus rules they derived optimality conditions for infinite and semi-infinite programming.

However, as far as we know, the literature does not provide an upper-estimate for the subdifferential of an arbitrary family of functions $\{f_t:t\in T \}$. This observation motivates our research to derive general  upper estimations for the subdifferential of the  supremum function under an arbitrary index set $T$ and without the uniform locally Lipschitz condition. {The aim of this work is to extend the results of \cite{MR3033113} and give general formulae for the subdifferential of the supremum function, in order to apply them  to derive necessary optimality conditions for general problems in the framework  of infinite programming. The main motivation for considering an arbitrary family of functions comes from the fact that indicators of sets are commonly used in variational analysis to study constraints and set-valued maps related with optimization problems (for example, stability of optimization problems and differentiability of set-valued maps) and they cannot, at least directly, be assumed to be locally Lipschitz. Furthermore, this approach allows us to also study the convex case, and recover general formulae in the convex case,   which in particular shows a unifying approach to the study of the subdifferential of the supremum function. For the sake of brevity, we will confine ourselves to extending the results of  \cite{MR3033113}, keeping in mind our applications for a future work.}

The rest of the paper is organized as follows: In \Cref{SECTION:NOTATION} we summarize the notation that we  use in this paper, which is classical in variation analysis.  In \Cref{SECTION:BASICPROPERTIES} we establish basic properties about the Fr\'echet subdifferential. We begin   \Cref{SECTION:SUBPOINTSUPREMUM} giving the definition of \emph{robust infimum} (see  \cref{robustinfimum}), this notion fits perfectly with our purpose. It can be understood as a bridge, which allows us to express the subgradient of the supremum function as \emph{robust  minimum} of perturbed functions, when the family $\{f_t:t\in T\}$ is an \emph{increasing family of functions}. Nevertheless, the  increasing property of the functions can be obtained considering the max functions over all finite sets of $T$ (see  \Cref{THEOREM:FORMULA:SUPREMUM}). In  \Cref{limitingSub}, where the main results are established, we study the  limiting subdifferential, this section is divided into  two subsections. First, we consider a finite-dimensional space; in this framework we establish a technical result (see  \Cref{Lema51}), which  can be applied to several results, but for simplicity we choose only one setting (see  \cref{teoremcompactindex}), where we provide a convex upper-estimation of the subdifferential. Second, we consider an infinite-dimensional Asplund space. This subsection starts with a result concerning a fuzzy calculus rule for  the normal cone of an intersection of an arbitrary family of sets (see  \cref{TEOCONES}). Later,  we use the definition of \emph{sequential normal epi-compactness} together with some results of \emph{separable reduction} to get   \cref{Mordukhovich:Separable}; this gives as a consequence a generalization of \cite[Theorem 3.2]{MR3033113} (see  \cref{TEO:MORD:NGH}), for non-necessarily uniformly Lipschitz functions. Finally, in \Cref{SECTION:CONVEXSUB} we apply our results to  the convex subdifferential, that is, when the functions $f_t$ are convex. In this section we get new results and also we  recover the general formula of Hantoute-L\'opez-Z\v{a}linescu \cite[Theorem 4]{MR2448918}.
	\section{Notation}\label{SECTION:NOTATION}

Throughout the paper and unless we stipulate  to the contrary, we adopt the following notation, $(X,\| \cdot \|)$ will be an Asplund space (i.e.,  every separable subspace of $X$ has separable dual) and $X^\ast$ its topological dual, with its norm  denoted by $\| \cdot \|_{\ast}$. The   bilinear form $ \langle \cdot ,\cdot \rangle :X^\ast\times  X \to \R$ is given by $\langle x^\ast,x \rangle :=x^\ast(x) $. The weak$^{\ast }$-topology on $X^\ast$ is  denoted by $w(X^{\ast },X)$ ($w^{\ast },$ for
short). The set of all convex, balanced and closed neighborhoods of a point $x$ with respect to the topology $\tau$ is denoted by $\mathcal{N}_x(\tau)$ ($\mathcal{N}_x$ for short). We will write $\Rex:= \R\cup\{-\infty,+\infty \}$ and we adopt the conventions  $1/\infty=0$, $0\cdot \infty= 0 = 0\cdot(- \infty)$ and $ \infty+( -\infty)=(-\infty)+\infty=\infty$. 

The closed unit ball in $X$ and $X^\ast$ are denoted by $\mathbb{B}$ and $\mathbb{B}^\ast$ respectively.  For a  point $x\in X$ (resp. $x^\ast \in X^\ast$) and a number $r\geq 0$ we set $\mathbb{B}(x,r):=x + r \mathbb{B}$ (resp. $\mathbb{B}^\ast(x^\ast,r)= x^\ast + r\mathbb{B}^\ast$).  For a function $f:X\to \Rex$ the set  $\mathbb{B}(x,f,r)$ is defined as the set of all $x' \in \mathbb{B}(x,r)$ such that $|f(x)- f(x')| \leq r$. The symbol $x' \overset{f}{\to}  x$ means $x' \to x$ and $f(x') \to f(x)$; we avoid  some misunderstandings about the topology $\tau$ considered in the last convergence using the  notation $x' \overset{\tau}{\to} x$ which emphasizes  that  the convergence $x' \to x$ is  with respect to the topology $\tau$. 

We denote by $\inte(A)$, $\overline{A}$, $\co(A)$ and  $\cco(A)$, the interior, the closure, the \emph{convex hull} and the \emph{closed convex  hull} of $A$, respectively. The \emph{affine subspace generated by $A$} is denoted by $\aff(A)$. The \emph{polar set} and  \emph{annihilator} of $A$ are defined by
\begin{align*}
A^{\circ}&:=\{ x^\ast\in X^\ast \mid \langle x^\ast,x\rangle \leq 1,\; \forall x\in A  \},\\
A^{\perp}&:=\{ x^\ast \in X^\ast \mid \langle x^\ast	 , x\rangle = 0,\; \forall x\in A  \},
\end{align*}
 respectively. The \emph{indicator}  function of  $A$ is defined as
$\delta_A(x) := 0$, if $x \in A$ and $\delta_A(x)=+\infty$, if $x\notin A$.

Let  $f:X\to \Rex$ be a lower semicontinuous (lsc) function  finite at $x$. Then
	\begin{align*}
\hat{\partial}f(x) :=& \{ x^\ast\in X^\ast \mid \liminf\limits_{h \to 0} \frac{f(x+h) -f(x) - \langle x^\ast, h\rangle}{\| h\|} \geq 0 \},\\
\end{align*}
is called  the  \emph{Fr\'echet (or regular) subdifferential} of $f$ at $x$.  

 The \emph{limiting (or Mordukhovich, or basic) subdifferential}  and the \emph{singular}  \emph{subdifferential} can be  defined as
\begin{align*}
\partial f(x) := &\{ w^\ast\text{-}\lim x^\ast_n : x_n^\ast \in \hat{\partial} f(x_n), \text{ and } x_n \overset{f}{\to}x \},\\
\partial^\infty  f(x) := & \{ w^\ast\text{-}\lim \lambda_n x^\ast_n : x_n^\ast \in \hat{\partial} f(x_n), \; x_n \overset{f}{\to}x \text{ and } \lambda_n \to 0^+ \},
\end{align*}
respectively (see, e.g., \cite{MR2191744,MR2144010,MR1491362,mordukhovich2018variational} for more details). 

If $|f(x)| =+\infty$, we set $\partial f(x):=\emptyset$ for any of the previous subdifferentials. It is important to recall that when  $f$ is convex proper  and lsc  all of these subdifferentials coincide with the classical  subdifferential of convex analysis $$\sub f(x):=\{ x^\ast \in X^\ast : \langle x^\ast , y -x \rangle \leq f(y) -f(x), \forall y \in X  \}.$$


For any set $A$, the Fr\'echet (or Regular) and  the limiting (or Mordukhovich, or basic)   normal cone of $A$ at $x$ are given by 
$\hat{N}(x,A)=\hat{\sub} \delta_A (x)$ and  ${N}(x,A)={\sub} \delta_A (x),$ respectively.

Consider a set $T$ and a  family of functions $\{ f_t \}_{t\in T} \subseteq {\Rex}^T$, we define the  supremum function $f: X\to \Rex$ by
\begin{align}\label{supremumfunction}
	f(x):=\sup_{t\in T}f_t(x), \;\forall x\in X
	\end{align} The symbol $\Pf(T)$ denotes the set of all $F\subseteq T$ such that $F$ is finite. For $F\in \Pf(T)$ we denote $f_F(x):=\max_{s\in F} f_s(x)$.

Following the notation of \cite{MR3033113}, $\R^{T}$ is defined as the space of all multipliers $\lambda=(\lambda_t)$ and $\tilde{\R}^T$ denotes the set of all $\lambda \in \R^{T}$ such that $ \lambda_t \neq0$ for finitely many $t\in  T$; by the symbol $\# \lambda $ we denote the cardinal number of $\sup \lambda $.  The \emph{generalized simplex on $T$} is the set $\Delta(T):=\{ \lambda \in  \tilde{\R}^{T}  :  (\lambda_t) \geq 0 \text{ and } \sum_{t\in T}\lambda_t =1 \}$.  For  a point $\bar{x}$ and $\epsilon\geq 0$,  the set of \emph{$\epsilon$-active indices at $\bar{x}$} is denoted by $T_\epsilon(\{f_t\}_{t\in T},\bar{x}):=\{  t \in T :  f(\bar{x}) \leq f_t(\bar{x}) + \epsilon \}$ ($T_\epsilon(\bar{x})$ for short), meanwhile the set of all \emph{$\epsilon$-active sets at $\bar{x}$} is denoted by $\mathcal{T}_{\epsilon}(\{f_t\}_{t\in T},\bar{x}):= \{ F \in \Pf(T) :  f(\bar{x}) \leq f_F(\bar{x}) + \epsilon \}$ ($\mathcal{T}_{\epsilon}(\bar{x})$ for short) and finally, we define 
\begin{align*}
\Delta(T,\{f_t\}_{t\in T},\bar{x},\epsilon):=\bigg\{ (\lambda_t) \in 	\tilde{\R}^{T}: \begin{array}{c}  \lambda_t \geq 0 \text{ for all } t\in T,\; \\ \lambda_t \leq \epsilon,\; \forall t \in T\backslash T_\epsilon(\bar{x})\\ \text{ and } | \sum_{t\in T} \lambda_t -1| \leq \epsilon	\end{array} \bigg \}
\end{align*} 
($\Delta(T,\bar{x},\epsilon)$ for short). When $T$ is a directed set ordered by $\preceq$, which means $(T,\preceq)$ is an ordered set and for every $t_1,t_2 \in T$ there exists $t_3 \in T$ such that $t_1 \preceq t_3$ and $t_2 \preceq t_3$,  we say that the family of functions is increasing provided that  for all $t_1,t_2 \in T$
\begin{align*}
t_1 \preceq t_2 \implies f_{t_1}(x) \leq f_{t_2}(x) , \; \forall x\in X.
\end{align*}   	

\section{Subdifferential of  supremum function}\label{sectionfuzzy}
In this section we establish  some fuzzy calculus rules for the Fr\'echet subdifferential of the supremum function. First we start \cref{SECTION:BASICPROPERTIES} recalling some basic properties of this subdifferential.   Posteriorly, we use the aforementioned properties to get fuzzy calculus rules for the supremum function of an arbitrary family of lower-semicontinuous functions.

\subsection{Basic properties of  the  Fr\'echet subdifferential}\label{SECTION:BASICPROPERTIES}
This section is devoted to stipulating some simple properties of the Fr\'echet subdifferentials. First, let us recall the following relation between the subdifferential and the normal cone to the epigraph of the function;  a point $x^\ast$ belongs to $\hat{\sub} f(x)$ if and only if  $(x^\ast , -1) \in  \hat{N}((x,f(x)),\epi f)$.

Now we write  the  next result, which is useful to understand Fr\'echet normal vectors to the epigraph of a function in terms of  subgradients in the Fr\'echet subdifferential, this result is well-known and we refer to  \cite{MR2159471,MR2191744,mordukhovich2018variational,MR1333396,MR1884910,MR2144010} for the proof.

\begin{proposition}\label{teo:singularaprox}
	Let $f : X \to \Rex$ be
	a proper lsc function and consider a point $(x^\ast, 0) \in \hat{N}(\epi f , (x,f(x))$. 
	Hence for any $\epsilon > 0$ there are points $y \in X$ and
	$(y^\ast, \lambda)\in \hat{N}(\epi f , (y,f(y))$ such that $\lambda \in (-\epsilon,0)$, $\| y - x \| \leq \epsilon$,  $|f(y) -  f(x)| <\epsilon$  and
	$y^\ast \in x^\ast + \epsilon \mathbb{B}^\ast$.
\end{proposition}
Next, we give some basic properties of the Fr\'echet subdifferentials. The first four properties are classicall in the literature, the final one can be proved using \cite[Theorem 3.1]{MR2965224} by  rewriting
a Fr\'echet subgradient satisfying an optimization problem as in \cite[Equation (3.8)]{MR3033113}. Nevertheless, we provide a proof for completeness.
\begin{proposition}\label{Prop:Properties}
The Fr\'echet subdifferential satisfies the following properties:
	\begin{enumerate}[label={\textnormal{P(\roman*)}},ref={\textnormal{P(\roman*)}}]
		\item\label{Property:4} Consider an lsc function $f:X\to \Rex$ and $x^\ast\in \hat{\sub} f(\bar{x})$. Then, for every $\epsilon>0$  there exists $\gamma >0$ such that the function 
		\begin{align*}
			x \to f(x)- \langle x^\ast , x -\bar{x} \rangle + \epsilon \| x -\bar{x}\| + \delta_{ \mathbb{B} (\bar{x},\gamma)} 
			\end{align*}
		 attains its minimum at $\bar{x}$.
		\item \label{Property:5}  (Calculus estimation) For every $\epsilon>0$,  any point $x\in X $  and every  finite-dimensional subspace $L$ of $X$, we have
		\begin{align*}
		\hat{\sub} \delta_{\mathbb{B}(x,\epsilon)\cap L} (x') \subseteq L^\perp, \; \forall x'\in \inte \mathbb{B}(x,\epsilon).
		\end{align*}
		\item\label{Property:3.5} (Enhanced Fuzzy Sum Rule) Consider an lsc function $f$, a convex Lipschitz function $g$ and a point  $x \in X$.  If $x$ is a local minimum of $f+g$ with  $f(x) \in \R$, there are sequences $(x_n,x_n^\ast)_{n \in\N  } $ such that  $x_n^\ast \in \hat{\sub} f(x_n)$,  $x_n \overset{f}{\to} x_0$,  $x^\ast_n\overset{\| \cdot \| }{\to} x^\ast_0$ with $  -x_0^\ast \in \hat{\sub} g(x)$. 
		\item \label{Property:3} (Fuzzy Sum Rule) Consider a  finite family of lsc functions $f_j :X \to \Rex$ with $j \in J$  and  $x^\ast\in \hat{\sub} (\sum_{j\in J} f_j)(x)$. Then,  there are nets $(x_{\alpha,j},x^\ast_{\alpha,j})_{\alpha \in \mathbb{D} }$  such that $ x^\ast_{\alpha, j}\in  \hat{\sub} f_j(x_{\alpha,j})$, $x_{\alpha,j} \overset{f}{\to} x$ and $\sum_{j\in J1} x_{\alpha,j}^\ast \overset{w^\ast}{\to} x^\ast$.
		\item\label{prop:v} For every finite family of lsc functions $f_j :X \to \Rex$ with $j \in J$ we have that  for all $x\in X$
		\begin{align}
		\hat{\sub}  f_J(x) \subseteq \bigcap\limits_{\epsilon>0} \cl^{w^\ast} \bigg\{  \sum \lambda_t \hat{\sub}  f_{j}(x_j) : \begin{array}{c} x_j \in \mathbb{B}(x,f_j,\epsilon), \lambda  \in \Delta(J, x,\epsilon)\\
		 \text{ and } \#\lambda\leq \dim(X) + 1  \end{array} \bigg\}.
		\end{align}
	\end{enumerate}
\end{proposition}

	\begin{proof}
	 \cref{Property:4,Property:5}  follow from definition.  \cref{Property:3.5} is the  well-known \emph{Enhanced Fuzzy Sum Rule} (see, e.g.,  \cite{MR1453305,MR1777630,MR2159471,MR2191744,Clarke:1998:NAC:274798}). \cref{Property:3} is an equivalence of the  \emph{Enhanced Fuzzy Sum Rule} (see, e.g.,  \cite{MR2840670}). Finally, we must prove  \cref{prop:v}; to complete this  task, it is enough to  consider the pointwise maximum of  two functions $g:=\max\{ f_1,f_2\}$. Let $x^\ast \in \hat{\sub} g(x)$, $\epsilon\in (0,1)$,  $V \in \mathcal{N}_{0}(w^*)$, so by \cref{Property:4} there exist $\gamma \in (0,\epsilon)$   such that the function $$y \to g(y) -\langle x^\ast , x -\bar{x} \rangle + \epsilon \| y - x\| + \delta_{ \mathbb{B}(x,\gamma)} (y)$$ attains its minimum at $x$. Hence, assuming that $\gamma >0$ is small enough, one can  suppose that
	\begin{align}\label{EQ:LSC}
	f_i(u) > f_i(x) - \epsilon, \text{ for all }u \in \mathbb{B}(x,\gamma),\; i=1,2.
	\end{align} 
	Now consider the function $$X\times \R^2 \ni (w,\alpha_1,\alpha_2) \to m(\alpha_1,\alpha_2)+ \delta_{\epi f_1 }(w,\alpha_1) + \delta_{\epi f_2}(w,\alpha_2)- \phi(w) + \delta_{F \cap \mathbb{B}(x,\gamma)}(w),$$ where $m(\alpha_1,\alpha_2):=\max\{\alpha_1,\alpha_2\}$ and $\phi(y):=\langle x^\ast , x -\bar{x} \rangle - \epsilon \| y - x\| $ . This function has a local minimum at the point $(x,f_1(x),f_2(x))$, so by \cref{Property:3} we can choose
	\begin{enumerate}[label={(\roman*)},ref={(\roman*)}]
		\item\label{item:i} $(\alpha_1,\alpha_2)   \in \R^2$ with $|f_i(x) - \alpha_i| \leq \gamma/2$ and $(q_1,q_2)\in \hat{\partial} m (\alpha_1,\alpha_2)=\{  (p_1,p_2) \in \Delta(\{ 1,2\}): p_i=0 \text{ if } \alpha_i < m(\alpha_1,\alpha_2) \}$.
		\item\label{item:ii} $(w_i,\beta_i)  \in \mathbb{B}(x, f_i(x),\gamma/2 )$  and $(w^{\ast}_i,\lambda_i) \in \hat{\sub} \delta_{\epi f_i} (w_i,\beta_i)$
	\end{enumerate}
	such that $w^\ast_1 + w^\ast_2 \in x^\ast + V + V$, $|q_1 + \lambda_1| < \gamma/2$ and $|q_2 + \lambda_2| < \gamma/2$. Consequently, by   \cref{EQ:LSC,item:ii} we have that $(w_i,f_i(w_i))  \in \mathbb{B}(x, f_i(x),\epsilon )$ by  classical argumentation we have that $(w^{\ast}_i,\lambda_i) \in \hat{\sub} \delta_{\epi f_i} (w_i,f_i(w_i))$  and $\lambda_i \leq 0$ (see, e.g., \cite{MR2159471,MR2191744,mordukhovich2018variational,Clarke:1998:NAC:274798}). Now, we check that $(-\lambda_1,-\lambda_2) \in \Delta(\{1,2\},x,\epsilon)$, indeed $|\lambda_1 + \lambda_2 - 1|=|\lambda_1 + \lambda_2 - q_1+q_2| \leq \epsilon$; moreover   if $f_i(x)<g(x)$ (for small enough $\epsilon$) we can assume (by \cref{item:i}) that $\alpha_i <m(\alpha_1,\alpha_2)$, so $q_i=0$ and consequently $|\lambda_i| \leq \epsilon$. Now, if  $\lambda^\ast_i\neq 0$ for $i=1,2$, we define $x_i^\ast:=- \lambda_i^{-1} w_i^\ast \in\hat{ \sub} f(w_i)$; otherwise if there exists some   $\lambda_i=0$, then one can approximate this element using   \cref{teo:singularaprox}. Therefore, we have proved that
		\begin{align*}
		\hat{\sub}  f_J(x) \subseteq \bigcap\limits_{\epsilon>0} \cl^{w^\ast} \bigg\{  \sum \lambda_t \hat{\sub}  f_{j}(x_j) : \begin{array}{c} x_j \in \mathbb{B}(x,f_j,\epsilon),\\ \lambda  \in \Delta(J, x,\epsilon)\end{array} \bigg\}.
	\end{align*}
	
	Now assume that $X$ is finite-dimensional. Consider $x^\ast= \sum_{i=1}^{k} \lambda_i x^\ast_i$ for some  $k >\dim(X)+1$ with $\lambda_i>0$, $x^\ast_i \in  \hat{\sub}  f_{t_i}(x_i)$,  $ x_i \in \mathbb{B}(x,f_{t_i},\epsilon)$ and $\lambda \in \Delta(J, x,\epsilon)$. Hence, $\{ (x^\ast_i,1)\}_{i=1}^k \subseteq X\times \R$ must be linearly dependent in $X \times \R$, and there are numbers $(\alpha_i)_{i=1}^k  \subseteq \mathbb{R}$ not all equal to zero such that $\sum_{i=1}^k \alpha_i x^\ast_i=0$ and $\sum \alpha_i =0$. Now consider 
	\begin{align}\label{CHOOSEOFBETA}
	\beta:=\min\{  \frac{ \lambda}{|\alpha_i|} : i \in I^+ \cup I^{-} \}, \text{where }I^+:=\{  i : \alpha_i >0\}  \text{ and }I^{-} := \{ i: \alpha_i <0		\} .
	\end{align}
	Then,
	\begin{enumerate}
		\item[1)] If $\beta= \frac{ \lambda_{i_0}}{\alpha_{i_0}}$ for some $i_0 \in I^+$, we notice that 
		$$x^\ast = \sum_{i=1}^{k} (\lambda_i- \beta \alpha_i) x^\ast_i=\sum_{\substack{i=1\\ i\neq i_0}}^{k} (\lambda_i- \beta \alpha_i) x^\ast_i,$$ moreover $|\sum_{i=1}^{k} (\lambda_i- \beta \alpha_i) - 1| = |\sum_{i=1}^{k} \lambda_i- 1| \leq \epsilon$ and for all  $t_{i} \notin T_\epsilon(x)$
		\begin{enumerate}
			\item[1.1)] If  $i \in I^+$, $0 \leq \lambda_i- \beta \alpha_i \leq \lambda_i \leq \epsilon$.
			\item[1.1)] If  $i \in I^{-}$, $0 \leq \lambda_i- \beta \alpha_i = \lambda_i +  \beta |\alpha_i|\leq 2\lambda_i \leq 2\epsilon$ (recall \cref{CHOOSEOFBETA}).
		\end{enumerate}
		\item[2)] If $\beta= \frac{ \lambda_{i_0}}{\alpha_{i_0}}$ for some $i_0 \in I^{-}$, we notice that 
		$$x^\ast = \sum_{i=1}^{k} (\lambda_i+ \beta \alpha_i) x^\ast_i=\sum_{\substack{i=1\\ i\neq i_0}}^{k} (\lambda_i+ \beta \alpha_i) x^\ast_i,$$ moreover $|\sum_{i=1}^{k} (\lambda_i + \beta \alpha_i) - 1| = |\sum_{i=1}^{k} \lambda_i- 1| \leq \epsilon$ and for all  $t_{i} \notin T_\epsilon(x)$
		\begin{enumerate}
			\item[2.1)] If  $i \in I^{-}$, $0 \leq \lambda_i+ \beta \alpha_i \leq \lambda_i \leq \epsilon$.
			\item[2.1)] If  $i \in I^+$, $0 \leq \lambda_i+ \beta \alpha_i = \lambda_i +  \beta |\alpha_i|\leq 2\lambda_i \leq 2\epsilon$ (recall \cref{CHOOSEOFBETA}).
		\end{enumerate}
	\end{enumerate}
	Therefore, $$x^\ast \in \big\{  \sum_{t\in J} \lambda_t \hat{\sub}  f_{t}(x_t) : \begin{array}{c}
	 x_t \in \mathbb{B}(x,f_t,2\epsilon), (\lambda_t) \in \Delta(J, x,2\epsilon)\\ \text{ and } \# (\lambda_t) \leq k-1 \end{array} \big\}.$$ 
	 
	 Repeating the processes (if $k -1 > \dim (X)+1$) one gets that $$x^\ast \in \big\{  \sum_{t\in J} \lambda_t \hat{\sub}  f_{t}(x_t) : \begin{array}{c}
 x_t \in \mathbb{B}(x,f_t,2^{p}\epsilon), (\lambda_t) \in \Delta(J, x,2^p\epsilon) \\ \text{ and } \# (\lambda_t) \leq \dim(X)+1  \end{array} \big\}$$ with $p=\#J - \dim(X)-1$. 
\end{proof}

	\subsection{Fuzzy calculus rules for the subdifferential of the supremum function}\label{SECTION:SUBPOINTSUPREMUM}
In this section $T$ will be an arbitrary index set and $f_t:X\to \Rex$ will be a family of lsc functions. We recall that  $f$ is defined as the supremum function of the family \cref{supremumfunction}.

 The next definition is an adaptation of the notion of the  \emph{robust infimum} or the \emph{decoupled infimum} used in subdifferential theory to get \emph{fuzzy calculus rules} (see, e.g., \cite{MR2191744,mordukhovich2018variational,MR2986672,MR3033176,MR2144010,Perez-Aros2018}).
\begin{definition}[robust infimum]\label{robustinfimum}
	We will say that the family $\{f_t : t \in T\}$ has a robust infimum on $B\subseteq X$ provided that 
	\begin{align}\label{robustinfimum:equation}
	\inf\limits_{x\in B} f(x) =\sup\limits_{t\in T}\inf\limits_{x \in B}f_t(x).
	\end{align}
	In addition, if there exists some $\bar{x}\in B$  such that $\sup\limits_{t\in T}\inf\limits_{x \in B}f_t(x)= f(\bar{x})$, then we will say that  $\{ f_t : t \in T \}$ has a robust  minimum on $B\subseteq X$. Finally, we say that    the family $\{ f_t : t \in T \}$ has a  robust local minimum at $\bar{x}$ if $\{ f_t : t \in T \}$ has a robust  minimum on some neighborhood $B$ of $\bar{x}$.
\end{definition}
The next lemma shows a sufficient  condition for the existence of a robust  minimum. {We recall that a function $g:X \to  \Rex$, where $(X, \tau)$ is a topological space, is called $\tau$-infcompact provided that for every $\alpha \in \R$ the sublevel set $\{ x \in X : g(x)\leq \alpha  \}$ is $\tau$-compact.} 
	\begin{lemma}\label{Lemma:SuficientCondition}[Sufficient condition for robust  minimum]
	Let $X$ be a Banach space and  $B\subseteq X$. Suppose that   $\{f_t : t \in T\}$ is an increasing family of $\tau$-lsc, 	$B$ is $\tau$-closed and there exists some $t_0$ such that $f_{t_0}$  is $\tau$-infcompact on $B$, with $\tau$ some topology coarser (weaker or smaller)  than the norm topology. Then the family $\{f_t : t \in T\}$ has a robust  minimum on $B$.
\end{lemma}
\begin{proof}
	\cite[Lemma 3.5]{Perez-Aros2018}
\end{proof}

{It is worth mentioning that in the above result the interchange between minimax in  \cref{robustinfimum:equation} is given without any convex-concave assumptions as in classical results (see, e.g.,  \cite{MR1921556,MR2144010,MR791361,MR838482,MR0055678,MR0312194}). This follows from the fact that in our result these assumptions are replaced by the  increasing property of the family of functions.}

\begin{remark}
	it has not escaped our notice  that the hypothesis of infcompactness of some $f_t$ is necessary, even if the supremum function $f$ is infcompact. Indeed, consider 
	$f_n(x) =n^2x^2-x^4$, then it is easy to see that $f_{n}\leq f_{n+1}$ and $f=\delta_{\{0\}}$; moreover $\inf_{\R} f_n =-\infty$ and $\inf_{\R} f=0$. 
\end{remark}
	The next results give us a necessary condition for the existence of robust  minimum in terms of an approximate Fermat's rule. More precisely, we have the following results
\begin{proposition}\label{TEO:RobustMininum}
Let  $\{f_t : t \in T \}$  be an increasing family of lsc functions. If  $\{f_t : t \in T \}$  has a robust local minimum at  $\bar{x}$,  then
	\begin{align}\label{FORM:1}
	0\in \bigcap\limits_{\epsilon >0} \cl^{\| \cdot \| }\bigg \{ \bigcup\{  \hat{\sub} f_t(x) : x\in \mathbb{B}(\bar{x},f_t,\epsilon), \; t \in T_{\epsilon}(\bar{x})  \}        \bigg\}.
	\end{align}
\end{proposition}
\begin{proof}
	Assume that  $\{ f_t : t \in T \}$  has a robust  minimum at  $\bar{x}$ on   $B:= \mathbb{B}(\bar{x},\eta)$. Pick  $ \epsilon \in (0,1)$ and  $\gamma \in (0,\min\{\eta/2,\epsilon/2\})$, since $\bar{x}$ is a  robust  minimum there exists some $t\in T$ such that $\inf\limits_{B} f_t \geq f(\bar{x})-\gamma^2 \geq f_t(\bar{x}) -\gamma^2$, so $| f_t(\bar{x}) - f(\bar{x})|\leq \gamma^2$ and $\bar{x}$ is a $\gamma^2$-minimum of $f_t +\delta_B$. Hence, by \emph{Ekeland's Variational Principle} (see, e.g., \cite{MR2144010}) there exists $x_\gamma \in \mathbb{B}(\bar{x}, \gamma)$ such that $| f_t(x_\gamma) -f_t(\bar{x})| \leq \gamma^2$ and $x_\gamma$ is a minimum of the function $f_t(\cdot)+\delta_B(\cdot) + \gamma \| \cdot - x_\gamma \| $, which implies that $f_t(\cdot)+ \gamma \| \cdot - x_\gamma \| $  attains a local minimum at $x_\gamma$. By  \cref{Prop:Properties}  \cref{Property:3.5} there exist sequences $(x_n,x_n^\ast) \in X\times X^\ast$ such that  $x_n^\ast \in  \hat{\sub} f_t(x_n)$,  $x_n \overset{f_t}{\to} x_\gamma$,  $x^\ast_n\overset{\| \cdot\|}{\to} \bar{x}^\ast$ with $  \bar{x}^\ast \in \gamma \mathbb{B}^\ast$. Then, take $n \in \N$ such that $|f_t(x_n) -f_t(x_\gamma)| \leq \gamma$, $\| x_n - x_\gamma\| \leq \gamma$ and  $0\in  \hat{\sub} f_t(x_n) + 2 \gamma \mathbb{B}^\ast $. Therefore, $ x_n \in \mathbb{B}(\bar{x},f_t, \epsilon)$,  $| f_t(\bar{x}) - f_t(x_n)| \leq \epsilon$, $| f(\bar{x}) - f_t(x_n)| \leq \epsilon$ and   $0\in  \hat{\sub} f_t(x_n) + \epsilon \mathbb{B}^\ast$;  to that end $0 \in \bigcup\{   \hat{\sub} f_t(x) : x\in \mathbb{B}(\bar{x},f_t(\bar{x}),\epsilon), \; t \in T_{\epsilon}(\bar{x})   \}    + \epsilon \mathbb{B}^\ast$.
\end{proof}
Now, we notice that, in particular, \cref{Lemma:SuficientCondition} shows that every minimum over a closed bounded set in a finite-dimensional space is necessarily a \emph{robust local minimum}. This fact, together with the representation of  \cref{Property:4}, helps us to understand the subgradients in terms of the definition of a \emph{robust  local minimum}. Also in an infinite-dimensional space, this compactness property can be forced using the $w^\ast$-topology. Consequently, we use \cref{TEO:RobustMininum} to give an upper-estimation of the subdifferential of the supremum function of an increasing family of functions.
\begin{proposition}\label{teo:sup:1}
	Let  $\{f_t : t \in T \}$  be an increasing family of lsc functions. Then for all $\bar{x} \in X$
	\begin{align}
	\hat{\sub} f(\bar{x}) \subseteq \bigcap\limits_{\epsilon >0} \cl^{w^\ast} \bigcup \bigg \{  	\hat{\sub} f_t(x) : x\in \mathbb{B}(\bar{x},f_t(\bar{x}), \epsilon), \; t\in T_{\epsilon}(\bar{x})        \bigg\}.
	\end{align}

\end{proposition}
\begin{proof}
	Fix $x^\ast \in 	\hat{\sub} f (\bar{x})$, $V\in \mathcal{N}_{0}(w^\ast)$, $\epsilon >0$ and $L $ a finite-dimensional subspace of $X$ such that $L^\perp \subseteq V$,  so by \cref{Property:4} there exist a ball  $B:=\mathbb{B}(\bar{x},\eta)$   {such that the function}
	$\tilde{f}:=f-\langle x^\ast , \cdot - \bar{x}\rangle + \epsilon \| \cdot - \bar{x}\|  +\delta_{L\cap B} $
	{ attains its minimum at $\bar{x}$.}
	
 Hence, consider the family of functions $\tilde{f}_t := f_t - \langle x^\ast , \cdot - \bar{x}\rangle + \epsilon \| \cdot - \bar{x}\|  +\delta_{L\cap B}$. It is easy to see that the family is increasing, $\tilde{f}=\sup_{T}\tilde{f}_t$ and there exists some $t \in T$ such that $\tilde{f_t}$ is infcompact. Whence,  \cref{Lemma:SuficientCondition} shows that the family $\{\tilde{f}_t: t\in T \}$ has a  robust local minimum at $\bar{x}$, and   \cref{TEO:RobustMininum} implies
	\begin{align}\label{teo:sup:1:eq}
	0\in \bigcap\limits_{\gamma >0} \cl^{w^\ast}\bigg \{ \bigcup\{  	\hat{\sub} \tilde{f}_t(x) : x\in \mathbb{B}(\bar{x},\tilde{f}_t,\gamma), \; t\in T_{\gamma}(\{\tilde{f}_t\}_{t\in T},\bar{x})   \}        \bigg\}.
	\end{align}
	Now take  $\nu \in (0, \min\{ \epsilon/3,\eta/3\} )$ small enough such that $| \phi(w)-\phi(\bar{x})| \leq \epsilon/3$ for all $w\in \mathbb{B}(\bar{x},\nu)$,  so by \cref{teo:sup:1:eq} there exist $t\in T_{\nu}(\{\tilde{f}_t\}_{t\in T},\bar{x})$,   $x \in \mathbb{B}(\bar{x},\tilde{f}_t,\nu) $ and $w^\ast \in  	\hat{\sub} \tilde{f}_t(x)=\hat{\sub} (f -\phi + \delta_{B\cap L})(x)$  such that  $w^\ast \in x^\ast + V$. This implies that $x\in \mathbb{B}(\bar{x},f_t,\nu+\epsilon/3)$ and $t\in T_{\nu+\epsilon/3}(\{f_t\}_{t\in T},\bar{x})$.
	
	Now applying  \cref{Prop:Properties} \cref{Property:3,Property:5}  to  $ \tilde{f}_t$ we get the existence of points $u \in X$ and $u^\ast\in X^\ast$ such that  $u^\ast \in  	\hat{\sub} f_t(u)$, $u \in \mathbb{B}(x,f_t,\nu)$  and $u^\ast  \in w^\ast + L^\perp +V=w^\ast + V$. Therefore $t\in T_{\epsilon}(\{f_t\}_{t\in T},\bar{x})$, $u \in \mathbb{B}(\bar{x},f_t,\epsilon)$ and $x^\ast  \in u^\ast  + V+ V$.
\end{proof}	

{Now we present a fuzzy calculus rule for a not necessarily increasing family of functions; we bypass this assumption using the family of finite sets of the index set $T$, which is always ordered by inclusion}.

\begin{theorem}\label{THEOREM:FORMULA:SUPREMUM}
	Let $\{f_t : t \in T \}$ be an arbitrary family of lsc functions. Then for every $\bar{x}\in X$
	\begin{align}\label{THEOREM:FORMULA:SUPREMUM:FORMULA}
	\hat{\sub} f(\bar{x})\subseteq \bigcap\limits_{\epsilon >0} \cl^{w^\ast}\bigg \{ \bigcup\limits_{ \substack{F\in \mathcal{T}_{\epsilon}(\bar{x})\\ x' \in \mathbb{B}(\bar{x},f_F,\epsilon) }	} \bigcap\limits_{\gamma>0} \cl^{w^\ast} \{  \sum\limits_{t \in F } \lambda_t  \hat{\sub} f_t(x_t) : \hspace{-0.2cm}\begin{array}{c}
	x_t \in \mathbb{B}(x', f_t,\gamma),\\ \lambda \in \Delta (F,x',\gamma) \text{ and}  \\\# \lambda\leq \dim(X)+1 
	\end{array} \}  \bigg\}
	\end{align}
\end{theorem}
\begin{proof}
	Consider the set $\tilde{T}:=\Pf(T)$, ordered by $F_1 \preceq F_2$ if and only if $F_1 \subseteq F_2$, and the family of functions $\{ f_F :F\in \tilde{T} \} $ (recall that $f_F=\max_{s \in F} f_s$), then it is easy to see that the family $\{ f_F  : F \in \tilde{T}\}$ is an increasing family of functions and $\sup_{F\in \tilde{T} } f_F   =f$.   Let $x^\ast\in  \hat{\sub} f(\bar{x})$, thus by  \cref{teo:sup:1}  $$x^\ast \in \bigcap\limits_{\epsilon >0} \cl^{w^\ast}\bigg \{ \bigcup\{   \hat{\sub} f_F(x') : x'\in \mathbb{B}(\bar{x},f_F,\epsilon), \; F \in \tilde{T}_{\epsilon}(\bar{x})   \}        \bigg\}.$$
	 Now, if $w^\ast \in   \hat{\sub} f_F(x')$ for some $x'\in \mathbb{B}(\bar{x},f_F,\epsilon)$ and  $F \in \tilde{T}_{\epsilon}(\bar{x}) $, we get  $x' \in \mathbb{B}(\bar{x},\epsilon) $ and  $F\in \mathcal{T}_{\epsilon}(\bar{x})$, so  using  \cref{Prop:Properties}  \cref{prop:v} we get 
	$$w^\ast\in\bigcap\limits_{\gamma>0} \cl^{w^\ast} \bigg\{  \sum \lambda_t \hat{\sub}  f_{t}(x_t) : \begin{array}{c} x_t \in \mathbb{B}(x',f_t,\gamma), \lambda  \in \Delta(F, x',\gamma)\\
	\text{ and } \#\lambda\leq \dim(X) + 1  \end{array} \bigg\},$$  then \cref{THEOREM:FORMULA:SUPREMUM:FORMULA} holds.
\end{proof}

{Here, it is important to compare the above result with  \cite[Theorem 3.1  part ii)]{MR3033113}. In the mentioned result, only uniform Lipschitz continuous data was considered. Here, we extend this fuzzy calculus  to arbitrary lsc data functions. Since the comparison between both results involves some technical estimations, we prefer to write this as a corollary.}
{
\begin{corollary}
	Under the hypothesis of \cref{THEOREM:FORMULA:SUPREMUM} assume that the data function $f_t$ is uniformly locally Lipschitz at $\bar{x}$. Then, for each $x^\ast \in \hat{\sub} f(\bar{x})$, $V \in \mathcal{N}_0(w^*)$ and $\epsilon>0$ there exist $\lambda \in \Delta (T_\epsilon(\bar{x}))$ and $x_t\in \mathbb{B}(\bar{x},\epsilon)$ for all $t\in T_\epsilon(\bar{x})$  such that 
	\begin{align}
		x^\ast \in \sum\limits_{t\in T_\epsilon(\bar{x})} \lambda_t  \hat{\sub} f_t(x_t) + V
	\end{align}
	\end{corollary} }
	\begin{proof}
		Consider $K$ as the constant of uniform Lipschitz continuity. Pick  $x^\ast \in \hat{\sub} f(\bar{x})$, and by \cref{THEOREM:FORMULA:SUPREMUM} we have that 
		\begin{align}\label{equation01}
				x^\ast \in \sum\limits_{t \in F } \lambda_t  \hat{\sub} f_t(x_t)  + V 
			\end{align}
			for some $F\in \mathcal{T}_{\epsilon}(\bar{x})$, a point $x' \in \mathbb{B}(\bar{x},f_F,\epsilon) $, points $x_t \in \mathbb{B}(x', f_t,\gamma)$ and $\lambda  \in \Delta (F,x',\gamma) $, we can assume that $\gamma \cdot \#F  \leq \epsilon$.   First $\| x_t - \bar{x} \|\leq \| x_t- x'\| +\| \bar{x} - x'\| \leq \epsilon+\gamma$.  Second $F_\epsilon(x') \subseteq T_{\epsilon (K+3)}(\bar{x})$, this is because
	\begin{align*}
		f_t(\bar{x}) &\geq f_t(x') - \epsilon K \geq f_F(x')  - \epsilon( K +1) \geq f_F(\bar{x}) - \epsilon( K +2)\\& \geq    f(\bar{x}) - \epsilon( K +3).
		\end{align*}
			Then, let us define $\tilde{\lambda} : T\to \mathbb{R}$ by
			\begin{align*}
				\tilde{\lambda}_t := \left\{ \begin{array}{cl}
					\frac{\lambda_{t}}{  \sum\limits_{t\in F_\gamma(x') } \lambda_t 	} &\text{ if } t\in F_\gamma (x') ,	\\
				0 , &\text{ otherwise. }
				\end{array}\right.
				\end{align*} 
		
			It is easy to see that $\tilde{\lambda} \in \Delta (T_{\epsilon(K+3)}(\bar{x}))$. Furthermore,  we claim that 
			\begin{align}\label{equation02}
				x^\ast \in  \sum_{t\in T} \tilde{\lambda}_t \partial f_t( x_t	)  +  3K\epsilon \mathbb{B} + V.
				\end{align}
			
		Indeed, by \cref{equation01}  there are $x^\ast_t \in \hat{ \partial} f_t(x_t)$ and $v^\ast \in V$ such that $x^\ast = \sum \lambda_t x_t^\ast  + v^\ast$, then
	\begin{align*}
		\| \sum\limits_{t\in T} \lambda_t x_t^\ast  - \sum\limits_{t\in T} \tilde{\lambda}_t x_t^\ast \|   = &	\| \sum\limits_{  t\in F_\gamma (x') } (\lambda_t  -  \tilde{\lambda}_t ) x_t^\ast  +  \sum_{  F \backslash F_\gamma (x') }  \lambda_t  x_t^\ast \|   \\
		\leq &    \bigg| \sum_{t\in F_\gamma(x')} \lambda_t  - 1 \bigg| K  + K \epsilon \leq   \bigg| \sum_{t\in F} \lambda_t  - 1 \bigg|K  + 2 \epsilon K\\\leq &  3K\epsilon.
		\end{align*}
	Consequently, \cref{equation02}. Finally, taking $\epsilon$ small enough we have that \cref{equation02} implies \cref{equation01}.
		\end{proof}

\section{Limiting subdifferential of pointwise supremum}\label{limitingSub}
This section is divided into two subsections. The first one concerns the study of the notion of the limiting subdifferential in finite-dimensional Banach spaces.  This setting is obviously motivated by the theory of \emph{semi-infinite programming}; in this scenario we  can obtain a better estimation of the limiting sequences obtained in \cref{THEOREM:FORMULA:SUPREMUM}. This result  is given in \cref{Lema51}; using this technical lemma, we focus on the particular case when the set $T$ is a subset of a compact metric space  (see \cref{teoremcompactindex}).  The second one corresponds to the infinite-dimensional setting; this subsection begins with a result concerning a \emph{fuzzy intersection rule for the normal cone of an arbitrary intersection of sets} (see \cref{TEOCONES}), which generalizes \cite[Theorem 5.2]{MR3000580}. Later the  main result of this subsection is given  in  \cref{Mordukhovich:Separable}, where we  explore the definition of \emph{sequential normal epi-compactness} (see, e.g., \cite{MR2191744}) and with this we  extend  \cite[Theorem 3.2]{MR3033113} (see \cref{TEO:MORD:NGH}).

\subsection{Finite-dimensional spaces}\label{limitingSubfinite}
 In this subsection $\hat{ \partial}$, $\partial$ and $\partial^\infty$ mean the Fr\'echet subdifferential, the limiting subdifferential  and the singular  limiting subdifferential, respectively. 
\begin{lemma}\label{Lema51}
	Consider  $\gamma_k \to 0$ and $x^\ast\in \partial f(x)$ and $y^\ast\in \sub^\infty f (x)$.  Then  there are sequences $\eta_k \to 0^+$, $\{t_{i,k}\}=F_k  \in \Pf(T)$, $\{t^\infty_{i,k}\}=F^\infty_k  \in \Pf(T)$  with  $\#F_k  \leq   \dim(X) + 1$, $\#F^\infty_k  \leq   \dim(X) + 1$,  $x'_k \to x$, $y'_k \to x$, $x_{i,k} \to x$, $y_{i,k} \to x$, $ \lambda_{i,k} \in \Delta(F_k, x'_k,\gamma_k)$, $ \lambda^\infty_{i,k} \in \Delta(F^\infty_k, y'_k,\gamma_k)$ such that:
	\begin{enumerate}[label={\roman*)},ref={\roman*)}]
		\item $x^\ast= \lim_{k\to \infty} \sum_{i\in F_k} \lambda_{i,k} \cdot  x^\ast_{i,k} $,  $y^\ast= \lim_{k\to \infty} \eta_k \sum_{i\in F_k} \lambda^\infty_{i,k} \cdot  y^\ast_{i,k} $,
		\item $\lim_{k\to \infty}  f_{F_k}(x_k') = f(x)$,  $\lim_{k\to \infty} f_{ F^\infty_k}(y_k') = f(x)$,
		\item $\lim\left| f_{t_{i,k}}(x_{i,k})- f_{t_{i,k} }(x_k') \right| =0$ and  $\lim\left| f_{t^\infty_{i,k}}(y_{i,k})- f_{t^\infty_{i,k} }(y_k') \right| =0$ for all $i$.
	\end{enumerate}

	Moreover (by passing to a subsequence) one of the following conditions holds.
	\begin{enumerate}[label={(\Alph*)},ref={(\Alph*)}]
		\item\label{Lema51Parta} There exists $n_1 \in \N$ with $n_1 \leq  \dim(X)+1 $ such that $\lambda_{i,k} \overset{k \to \infty}{\longrightarrow} \lambda_i > 0$, $x^\ast_{i,k} \overset{k \to \infty}{\longrightarrow} x_i^\ast$, $\lim f_{t_{i,k}}(x_{i,k}) = f(x)$ for $i \leq n_1$   and $\lambda_{i,k} \overset{k \to \infty}{\longrightarrow}  0$, $\lambda_{i,k}\cdot x^\ast_{i,k} \overset{k \to \infty}{\longrightarrow} x_i^\ast$ for $n_1 < i \leq n$, 
		
		and  $x^\ast=\sum\limits_{i=1}^{n_1} \lambda_i x^\ast_i + \sum\limits_{i>n_1}^{n} x_i^\ast$, or
		\item \label{Lema51Partb} There are $\nu_k \to 0$ such that $\nu_k \cdot \lambda_{i,k}\cdot x^\ast_{i,k} \overset{k \to \infty}{\longrightarrow} x_i^\ast$ and $\sum\limits_{i=1}^{n_1}x^\ast_i =0$ with not all $x_i^\ast$ equal to zero.
		
	\end{enumerate}
	and (up to a subsequence) one of the following conditions holds.
	\begin{enumerate}[label={(\Alph*$^\infty$)},ref={(\Alph*$^\infty$)}]
		\item\label{Lema51Partainf} There exists $n_2 \in \N$ with $n_2 \leq  \dim(X)+1 $ such that $\lambda^\infty_{i,k} \overset{k \to \infty}{\longrightarrow} \lambda^\infty_i > 0$, $y^\ast_{i,k} \overset{k \to \infty}{\longrightarrow} y_i^\ast$, $\lim f_{t^\infty_{i,k}}(y_{i,k}) = f(x)$ for $i \leq n_2$   and $\lambda^\infty_{i,k} \overset{k \to \infty}{\longrightarrow}  0$, $\lambda^\infty_{i,k}\cdot y^\ast_{i,k} \overset{k \to \infty}{\longrightarrow} y_i^\ast$ for $n_2 < i \leq n$, 
		
		and  $y^\ast=\sum\limits_{i=1}^{n_2} \lambda^\infty_i y^\ast_i + \sum\limits_{i>n_2}^{n} y_i^\ast$, or
		\item \label{Lema51Partbinf} There are $\nu_k \to 0$ such that $\nu_k \cdot\eta_k\cdot \lambda^\infty_{i,k}\cdot y^\ast_{i,k} \overset{k \to \infty}{\longrightarrow} y_i^\ast$ and $\sum\limits_{i=1}^{n_1}y^\ast_i =0$ with not all $x_i^\ast$ equal to zero.
	
	\end{enumerate}
\end{lemma} 

	\begin{proof}
	Define $N:=\dim(X)+1$ and consider $x^\ast\in  \sub f(x)$ ($y^\ast \in \sub^\infty f(x)$, resp.), so (by definition) there exist $x_k \overset{f}{\to} x$  and $x_k^\ast \in \hat{ \partial} f(x_k)$ ($y_k \overset{f}{\to} x$, $\eta_k$ and $y_k^\ast \in \hat{ \partial} f(y_k)$, resp.) such that $x_k^\ast\to x^\ast$ ($\eta_k y_k^\ast\to y^\ast$, resp.). Whence, by  \cref{THEOREM:FORMULA:SUPREMUM}, there exist  $x'_k \in \mathbb{B}(x_k,\gamma_k)$ and $F_k=\{ t_{i,k} \}_{k=1}^N\subseteq T$, with $|f_{F_k}(x_k') -f(x_k)|\leq \gamma_k$ along with elements  $x_{{t_{i,k}}} \in \mathbb{B}(x'_k, f_{t_{i,k}},\gamma_k)$ and $z_k^\ast =\sum_{i=1}^N \lambda_{{t}_{i,k}} x^\ast_{ t_{ i,k }}$ with $\| z_k^\ast -x_k^\ast\|_{\ast} \leq \gamma_k$, $(\lambda_{k,i})\in  \Delta(F_k, x'_k,\gamma_k)$ and $x^\ast_{t_{i,k} } \in \hat{\sub} f_{t_{i,k}} (x_{i,k})$. Hence, $x^\ast= \lim\limits_{k\to \infty} \sum_{i\in F_k} \lambda_{i,k} \cdot  x^\ast_{i,k}$, $	\lim\limits_{k\to \infty}  f_{F_k} (x_k') = f(x)$ and $	\lim\limits_{k\to \infty} \left(  f_{t_{i,k}}(x_{i,k})- f_{t_{i,k} }(x_k') \right) =0$. Similarly, for the case $y^*\in { \partial}^\infty f(x)$, there exist $y'_k \in \mathbb{B}(y_k,\gamma_k)$ and $F^\infty_k=\{ t^\infty_{i,k} \}_{k=1}^N\subseteq T$, with $|f_{F^\infty_k}(x_k') -f(y_k)|\leq \gamma_k$ along with elements  $y_{{t_{i,k}}} \in \mathbb{B}(y'_k, f_{t_{i,k}},\gamma_k)$ and $w_k^\ast =\sum_{i=1}^N \lambda^\infty_{{t}_{i,k}} \eta_k y^\ast_{ t_{ i,k }}$ with $\| w_k^\ast -y_k^\ast\|_{\ast} \leq \gamma_k$, $(\lambda^\infty_{k,i})\in  \Delta(F^\infty_k, y'_k,\gamma_k)$ and $y^\ast_{t_{i,k} } \in \hat{\sub} f_{t^\infty_{i,k}} (y_{i,k})$.

	Now, we focus on the case $x^* \in \sub f(x)$; by passing to a subsequence, we have that $\lambda_{i,k} \to \lambda_i$ with $(\lambda_i) \in \Delta(\{1,...,N\})$ and (relabeling  it  if necessary) we may assume that $\lambda_k\neq 0$ for all $i=1,..,n_1$ and $\lambda_k=0$ for all $i=n_1 +1,...,N$. 
	
	On the one hand if $\sup\{  \| \lambda_{i,k} x^\ast_{i,k} \|_{\ast} : i=1,...,N;\;  k\in \N\}<+\infty$ (up to a subsequence) we can assume that $ \lambda_{i,k} x^\ast_{i,k}  \to \lambda_{i} x^\ast_i$ for all $i=1,...,n_1$ and $\lambda_{i,k} x^\ast_{i,k} \to x^\ast_i$ for all $i=n_1 + 1,..., N$, therefore $x^\ast=\sum\limits_{i=1}^{n_1} \lambda_i x^\ast_i + \sum\limits_{i>n_1}^{n} x_i^\ast$. Next,  we claim that $\lim f_{t_{i,k}}(x_{i,k}) = f(x)$  for  all $i=1,...,n_1$.  Indeed, define $\gamma:= \min\{ \lambda_{i}/2 : i=1,...,n_1 \}$, then for all $k$ (large enough) such that $\gamma_k \leq \gamma$ and $\lambda_k >\gamma$  (recall $t_{k,i}\in  \Delta(F_k,x'_k, \gamma_n)$) we have that
	$$ f_{t_{i,k}}(x_k')  + \gamma_k \geq \max_{s \in F_k} f_s(x_k')\geq f_{t_{i,k}}(x_k'),$$
	so, taking   the limits we obtain that  
	$$\lim_{k\to \infty} f_{t_{i,k}}(x_k') \geq \lim_{k\to \infty} \max_{s \in F_k} f_s(x_k')=f(x) \geq \lim_{k\to \infty} f_{t_{i,k}}(x_k'),$$
	which implies the desired conclusion.
	
	On the other hand, if   $\sup\{  \| \lambda_{i,k} x^\ast_{i,k} \|_{\ast} : i=1,...,N;\;  k\in \N\}=+\infty$ (by passing to a subsequence)  $\eta_k := \left( \max\limits_{i=1,...,k}  \| \lambda_{i,k} x^\ast_{i,k} \|_{\ast} \right) ^{-1} \to 0$ and  (w.l.o.g.) $\eta_k \lambda_{i,k} x^\ast_{i,k} \to x^\ast_i$ for all $i=1,...,N$, which implies that   $\sum\limits_{i=1}^{n_1}x^\ast_i =0$ with not all $x_i^\ast$ equal to zero.
	
	The case $y^\ast \in \sub^\infty f(x)$ follows similar arguments, so we omit the proof.
\end{proof}

Now we are going to apply the above result to a framework, where the  functions $f_t$'s represent a control in a region. We assume that $T$ is contained in a metric space and $\overline{T}$ is compact. For this reason we introduce the following definitions.

 A family of lsc  functions    $\{ f_t :   t\in T \}$  is said to be  \emph{continuously subdifferentiable at $x$} with respect to $\hat{ \partial}$ provided that for every sequence $T\times X\times [0,+\infty) \ni(t_n, x_n,\lambda_n) \to (t,x,\lambda) \in T\times X\times [0,+\infty)$ and points $w_n^\ast \in \hat{\partial} f_{t_n}(x_n)$ with $\lambda_n w^\ast_n \to w^\ast$ one has 
\begin{align*}
	w^\ast\in 	\lambda \circ \partial f_t(x) :=\left\{ \begin{array}{cl}
		\lambda \partial f_{t}(x) &\text{ if } \lambda >0,\\
		\partial^\infty f_{t}(x)&\text{ if } \lambda =0,
	\end{array}\right.
\end{align*}
{To our knowledge, the next definition was introduced  in \cite{MR2384963}, where the authors studied generalized notions of differentiation for  parameter-dependent set valued maps and mappings}. For a point $x \in X$ and $t\in \overline{T}\backslash T$	 we define the \emph{extended  subdifferential} and  the \emph{extended singular subdifferential} at $(t,x)$ as
\begin{align*}
	\partial f_t (x) :=& \bigg\{ x^\ast \in X^\ast :\begin{array}{c} 
		\exists t_k  \in T,\, t_k \to t,\, x_{k } \to x,\;x_k^\ast \in \hat{ \partial} f_{t_k}(x_k) \\
		\text{ s.t. } f_{t_k}(x_k)  \to f(x),\; \text{ and } x^\ast_k \to x^\ast
	\end{array} \bigg\},\\
	\partial^\infty f_t (x) :=& \bigg\{ x^\ast \in X^\ast :\begin{array}{c}
		\exists t_k  \in T,\, t_k \to t,\, \eta_k \to 0^+,\;  x_{k } \to x,\;x_k^\ast \in \hat{ \partial} f_{t_k}(x_k)\\
		\text{ s.t. } \limsup f_{t_k}(x_k)  \leq f(x),\; \text{ and } \eta_k x^\ast_k \to x^\ast
	\end{array} \bigg\},
\end{align*}
respectively.  Finally, we denote the \emph{extended active index set at $x$} by $\overline{T}(x)=T(x) \cup (\overline{T}\backslash T)$.

\begin{theorem}\label{teoremcompactindex}
	Consider  a family of lsc  functions $\{ f_t :   t\in T \}$ where $T$ is a subset of a metric space and $\overline{T}$ is compact. Assume that the following conditions hold at a point $\bar{x}$
		\begin{enumerate}[label={(\alph*)},ref={(\alph*)}]
		\item\label{Teo2a} For every $\bar{t} \in T$, $\limsup\limits_{(t,x) \to (\bar{t},\bar{x}) }  f_t(x) \leq f_t(\bar{x})$.
		\item\label{Teo2b}  The family is $\{ f_t :   t\in T \}$ continuously subdifferentiable at $\bar{x}$.
		\item\label{Teo2c}  The set $\co \left(  \bigcup_{t\in\overline{T}} \sub^\infty f_t(\bar{x}) \right) $ does not contain lines. 
	\end{enumerate}
	Then 
	\begin{align*}
	\sub f(\bar{x}) \subseteq&  \co \bigg( \bigcup\limits_{t \in \overline{T}(\bar{x}) }\sub f_t(\bar{x}) \bigg)  +\co \bigg( \bigcup\limits_{t \in \overline{T}} \sub^\infty f_t(\bar{x}) \bigg), \text{ and}\\
	\sub^\infty f(\bar{x}) \subseteq   &   \co \bigg( \bigcup\limits_{t \in \overline{T}} \sub^\infty f_t(\bar{x}) \bigg).
	\end{align*} 
\end{theorem} 
\begin{proof}
	Consider $x^\ast\in  \sub f(\bar{x})$. Now, using the notation of   \cref{Lema51} and by the compactness of $\overline{T}$ we can assume that $t_{k,i} \to t_i \in \overline{T}$.  Moreover,  \cref{Teo2c} contradicts \cref{Lema51} \cref{Lema51Partb,Lema51Partbinf}, which means,  \cref{Lema51} \cref{Lema51Parta,Lema51Partainf} must hold. Hence we can write $x^\ast=\sum\limits_{i=1}^{n_1} \lambda_i x^\ast_i + \sum\limits_{i>n_1}^{n} x_i^\ast$.
	
	\begin{itemize}
		\item If  $i\leq n_1$ and   $t_i \in T$:  By assumption \cref{Teo2a} and \cref{Lema51} \cref{Lema51Parta} necessarily $f(\bar{x}) =f_{t_i}(\bar{x})$, i.e., $t\in T(\bar{x})$. Also,  \cref{Teo2b}  implies $x^\ast_i \in \sub f_{t_i}(\bar{x})$. 
		\item If $i\leq n_1$ and  $t_i \in \overline{T}\backslash T$: By   \cref{Lema51} \cref{Lema51Parta} we get  that $x_i^\ast\in \sub f_{t_i}(\bar{x})$.
		\item If $i > n_1$ and  $t_i \in T$: By assumption \cref{Teo2b}  we get $x^\ast_i \in \sub f_{t_i}(\bar{x})$. 
		\item If $i > n_1$ and    $t_i \in \overline{T}\backslash T$: By   \cref{Lema51} \cref{Lema51Parta} implies that $x_i^\ast\in \sub f_{t_i}(\bar{x})$.
		\end{itemize}
	This completes the first part. The case  $y^\ast\in \sub^\infty f(\bar{x}) $ follows similar arguments so we omit the proof.
\end{proof}
{
It is important to mention that similar results have been shown in the literature; we refer to \cite{MR3205549,MR2384963,MR1058436} for some examples. In the above result we did not go for the greater stage of generality, and we established the result only to show one possible application of \cref{Lema51}.} 
{\begin{remark}
	It has not escaped  our notice that the  convex envelope appears in \cref{teoremcompactindex} due to the fact that at the moment of taking the convergent subsequence in the index $t_{k,i} \to t_i$ we cannot ensure, in a general framework, that there could exist two limit points $t_i=t_j$ for $i\neq j$. Nevertheless, the reader can force  this condition imposing some assumptions over the index set, the simplest example is when the index set is finite.
\end{remark}}

{Now let us finish this subsection  with an example which shows an application of \cref{teoremcompactindex} for a countable number of functions.}
\begin{example}
	Consider $T=\N$ and the sequence of functions $$f_n(x,y)=\left\{ \begin{array}{cl}
	nx^2 + \frac{n}{n-1}\log(|y|+1) - \frac{1}{n}&\text{ if } x\geq 0,\\
	\frac{n}{n-1}\log(|y|+1) - \frac{1}{n} & \text{ if } x< 0.
	\end{array}\right.	$$
{Here,	it is worth noting that all  functions $f_n$ are locally Lipschitz continuous, but they are not uniformly Lipschitz continuous, so the results of \cite{MR3033113} cannot be applied. Nevertheless, we can apply \cref{teoremcompactindex}.} Indeed,  after some calculus, we get that
	\begin{align*}
		\sub f_n(0,0)&=\{0\}\times  [- \frac{n}{n-1}, \frac{n}{n-1} ],\\
		\sub^\infty f_n(0,0) &= \{(0,0) \}.
		\end{align*}
	
We compute  the function $$f(x,y)= \log(|y|+1) + \delta_{(-\infty, 0]}(x) = \left\{ \begin{array}{cl}
	+\infty  &\text{ if } x> 0,\\
	\log(|y|+1)  &\text{ if } x\leq 0,
	\end{array}\right.	.$$ Then, $\sub f(0,0)=[0,+\infty)\times [-1,1]$ and $\sub^\infty f(0,0)=[0,+\infty)\times \{(0,0)\}$. 
	 In order to apply  \cref{teoremcompactindex}  we notice that  $\N$ is a subset of the compact space $\N_\infty :=\N\cup\{ \infty \}$ with the metric $d(a,b)=| \frac{1}{a} - \frac{1}{b} |$. Straightforwardly the assumptions  \cref{Teo2a,Teo2b} of \cref{teoremcompactindex} are satisfied, furthermore,   $\N(0,0)=\emptyset$. 
	 
	 Now, we calculate $ \partial  f_\infty (0,0)$ and $ \partial^\infty  f_\infty(0,0)$. First we notice that
	 \begin{align*}
	 	\hat{ \partial} f_{n}(x,y) \subseteq [0,+\infty) \times [- \frac{n}{n-1}, \frac{n}{n-1} ].
	 \end{align*}
	 Then $  \partial  f_\infty (0,0) =[0,+\infty)\times [-1,1] $ and $  \partial^\infty  f_\infty (0,0) =[0,+\infty)\times \{ 0\} $. In particular,    assumption \cref{Teo2c} of  \cref{teoremcompactindex} holds. Then,  \cref{teoremcompactindex} gives us 
	 
	 \begin{align*}
		\partial  f (0,0) &=\co \big( \partial  f_\infty (0,0)   \big) +\co \big(\bigcup\limits_{n\in \N_\infty } \partial^\infty  f_{n} (0,0)   \big)=[0,+\infty)\times [-1,1],  \\
		\partial^\infty  f (0,0) &=\co \big(\bigcup\limits_{n\in \N_\infty } \partial^\infty  f_{n} (0,0)   \big)=[0,+\infty)\times \{ 0\},
	\end{align*}
	 which are  exact estimations of the limiting and singular subdifferential of the function $f$ at $(0,0)$.

\end{example}

\subsection{Infinite-dimensional spaces} \label{limitingSubinfinite}

In this section we study the limiting subdifferential of the supremum function in an arbitrary Asplund space  $X$.

The first  result of this Subsection generalizes the \emph{Fuzzy Intersection Rule for Fr\'echet Normals to Countable Intersections of Cones} established in \cite[Theorem 5.2]{MR3000580}.

\begin{theorem}\label{TEOCONES}
	Let $\{ \varLambda_t\}_{t\in T}$ be an arbitrary family of closed subsets of $X$ and $\varLambda:=\bigcap\limits_{t \in T }\varLambda_t$.   Then given $\bar{x}\in X$, $x^\ast \in \hat{N}(\varLambda,\bar{x})$,  $\epsilon>0$ and  $V\in \mathcal{N}_0(w^\ast)$ there are $F \in  \Pf(T)$, $w_t \in \mathbb{B}(\bar{x},\epsilon)$ and $w^\ast_t \in \hat{N}(\varLambda_t, w_t)$ such that  
	\begin{align}\label{NORMALCONES0}
	x^\ast \in \sum_{t \in F} w^\ast_t + V.
	\end{align}
	Consequently, if $\{ \varLambda_t\}_{t\in T}$ is a family of closed cones $\hat{N}(\varLambda_t, w_t) \subseteq N(\varLambda_t, 0)$ for all $t\in T$ and 
	\begin{align}\label{NORMALCONES}
	\hat{N}(\varLambda,\bar{x}) \subseteq \cl^{w^\ast}\bigg\{ \sum\limits_{t\in F} w^\ast_t	 \bigg\arrowvert   w^\ast_t \in N(\varLambda_t, 0) \text{ and } t\in F\in \Pf(T)		\bigg\}.
	\end{align}
\end{theorem}
\begin{proof}
	The first part corresponds to a straightforward application of  \cref{THEOREM:FORMULA:SUPREMUM}. Now if one considers a closed cone $K\subseteq X$ and $u \in K$ one has that 
	\begin{align*}
	\hat{N}(K, u)  \subseteq \hat{N}(K, n^{-1}u),\; \forall n\in \N.
	\end{align*}
	Therefore $\hat{N}(\varLambda_t, u)  \subseteq N(\varLambda_t, 0)$ for every $t\in T$ and $u \in \varLambda_t$, consequently \cref{NORMALCONES0} implies \cref{NORMALCONES}.
\end{proof}
\begin{remark}
	It important to notice that  the results of \cite{MR3561780} cannot be applied to derive the above formulae, since imposing uniform Lipschitz continuity of an indicator function of the set $\Lambda$ at a point $\bar{x}$ is equivalent to assume that the point $\bar{x}$ is an interior point of  $\Lambda$, which give us a trivial conclusion.
\end{remark}

The next definition is the notion of  \emph{sequential normal epi-compactness} (SNEC) of functions  defined for the limiting subdifferential (see, e.g., \cite[Definition 1.116 and Corollary  2.39]{MR2191744}). 
\begin{definition}
	A real extended valued function $f$ finite at $x$ is said  to be  SNEC at $x$  if for any sequences	 $(\lambda_k, x_k, x_k^\ast) \in [0,+\infty) \times X\times X^\ast$ satisfying $\lambda_k \to 0$, $x_k \overset{f}{\to} x$, $x^\ast_k\in \hat{\sub} f(x_k)$ and $\lambda_k x^\ast_k \overset{*}{\rightharpoonup} 0$ one has $\| \lambda_k x^\ast_k\| \to 0$. A family of functions $\{f_t\}_{t\in T}$ is said to be SNEC on a neighborhood of a point $\bar{x}$ if there exists  a neighborhood $U$ of $\bar{x}$ such that for all $x\in U$   all but one of these are SNEC at $x$.
\end{definition}

	We  say that the family of functions $\{f_t: t\in T\}$ satisfy the \emph{limiting condition} on a neighborhood of a point $\bar{x}$ if there exists  a neighborhood $U$ of $\bar{x}$ such that for all
all $x \in U$ and  $F\in \Pf(T)$ 
\begin{align}\label{LIMITING:CONDITION} 
	  w^\ast_t \in \sub^\infty f_t(x), \; t\in F \text{ and } \sum_{t\in F} w^\ast_t =0 \text{ implies  } w^\ast_t =0, \text{ for all } t \in F.
\end{align}
 
 {It is worth mentioning that the SNEC property is immediately satisfied if the space $X$ is finite-dimensional. Moreover,  the family of functions $\{f_t\}_{t\in T}$ is  SNEC and satisfies the \emph{limiting condition} on a neighborhood of a point $\bar{x}$, provided that the functions are locally Lipschitz (not necessarily uniform) on a neighborhood $U$ of $\bar{x}$.}

The next theorem corresponds to the main result of this paper; in this result we give an upper-estimation of the subdifferential of the supremum function only using the above definitions, without the assumption of uniformly locally Lipschitz continuity.

\begin{theorem}\label{Mordukhovich:Separable} 
	Consider  a family of lsc  functions $\{ f_t :   t\in T \}$. If  the family $\{f_t:t\in T\}$ is SNEC   and satisfy the limiting condition \cref{LIMITING:CONDITION}  on a neighborhood of  $\bar{x}$.
	Then 
	\begin{align}\label{Mordukhovich:Separable01}
		\partial f(\bar{x}) &\subseteq \bigcap\limits_{\epsilon >0} \cl^{w^\ast}  \bigg(     \mathcal{S}(\bar{x},\epsilon)\bigg), \text{ and }
		\partial^\infty f(\bar{x}) \subseteq \bigcap\limits_{\epsilon >0}  \cl^{w^\ast} \bigg( [0,\epsilon] \cdot   \mathcal{S}(\bar{x},\epsilon)\bigg).
	\end{align}
	Where 
	\begin{align}\label{defn:Sxe}
		\mathcal{S}(\bar{x},\epsilon):= \left\{  \sum\limits_{t \in F } \lambda_t \circ \sub f_t(x') : \begin{array}{c}
			F\in \Pf(T), x' \in \mathbb{B}(\bar{x},\epsilon),  \\  |f_F(x')- f(\bar{x})| \leq \epsilon, \; \lambda\in \Delta(F) \\ \text{and } 
			f_{t}(x')=f_F (x') \text{ for all } t' \in \supp \lambda
		\end{array} \right\},
		\end{align}
	
	and \begin{align*}
		 	\lambda \circ \partial f_t(x) :=\left\{ \begin{array}{cl}
			\lambda \partial f_{t}(x), &\text{ if } \lambda >0,\\
			\partial^\infty f_{t}(x),&\text{ if } \lambda =0.
		\end{array}\right.
	\end{align*}
\end{theorem}
\begin{proof}
	Consider $\epsilon >0$ and $V \in \mathcal{N}_0(w^\ast)$. Pick $x^\ast\in \sub f(\bar{x})$ ($y^*\in \sub^\infty f(\bar{x} )$, resp.). Hence,  there exist  sequences $x_j \overset{f}{\to }\bar{x}$ and $x^\ast_j\overset{w^\ast}{\rightarrow} x^\ast$ ($\nu_j \to 0^+$ and $\nu_j x^\ast_j\overset{w^\ast}{\rightarrow} y^\ast$, resp.) with $x_j^\ast \in \hat{\sub} f(x_j)$.   Now,  take $j_0 \in \N$ such that $x^\ast \in  x^\ast_{j_0} + V$ ($x^\ast \in  \nu_{j_0}x^\ast_{j_0} + V$ and $\nu_{j_0} \leq \epsilon$, resp.) and $x_{j_0} \in \mathbb{B}(\bar{x},f,\epsilon)$. Hence,  by  \cref{THEOREM:FORMULA:SUPREMUM} there exist some $F\in \mathcal{T}_{\epsilon}(x_{j_0})$ and $x' \in \mathbb{B}(x_{j_0},f_F,\epsilon)$ such that  $x^\ast_{j_0}=w^\ast + v^\ast$ with 
	\begin{align*}
		w^\ast\in \bigcap\limits_{\gamma>0} \cl^{w^\ast} \{  \sum\limits_{t \in T } \lambda_t \sub f_t(x_t) : x_t \in \mathbb{B}(x', f_t,\gamma), \; (\lambda_t) \in \Delta (F,x',\gamma)  \}, 
	\end{align*} 
	and $v^\ast \in V$.  One gets   $x' \in \mathbb{B}(\bar{x},2\epsilon)$ and $|  f_F (x') - f(\bar{x})| \leq 3\epsilon$.  Now, we show that
	\begin{align}\label{eq:inclusion}
		w^\ast \in \mathcal{S}(\bar{x},3\epsilon)
	\end{align}

	For this purpose let us  introduce the following notation; by the symbol $S(X\times X^\ast)$ we understand the family of set $U\times Y$ where $U$ and $Y$ are (norm-) separable closed linear  subspaces of $X$ and $X^\ast$, a set $\mathcal{A} \subseteq S(X\times X^\ast)$ is called a \emph{rich family} if $(i)$ for every $U\times Y\in S(X\times X^\ast)$, there exists $V\times Z \in \mathcal{A}$ such that $U \subseteq V$ and $Y \subseteq Z$, and $(ii)$ $\overline{ \bigcup_{n\in \N } U_n}\times \overline{ \bigcup_{n\in \N } Y_n} \in \mathcal{A}$, whenever the sequence $(U_n\times Y_n)_{n\in  \N} \subseteq \mathcal{A}$ satisfies  $U_n \subseteq U_{n+1}$ and $Y_n \subseteq Y_{n+1}$ (see, e.g., \cite{MR3582299,MR3533170} and the references therein). We claim that under our assumptions there exists a rich family $\mathcal{A}$ such that for all $V \times Y$ and any sequence $y_n^\ast \in Y$ with $y_n^\ast \overset{w^\ast}{ \rightarrow} v^\ast$ and $v^\ast$ is zero on $V$, then $v^\ast$ is zero in the whole $X$.  Indeed, by \cite[Theorem 13]{MR3582299} there exists a rich family $\mathcal{A}\subseteq S(X\times X^\ast)$ such that for every 
	$ \mu:=V\times Y \in \mathcal{A}$ there exists a projection $P_\mu : X^\ast \to X^\ast$ satisfying that  $P_\mu(X^\ast)=Y$,  $P_\mu^{-1}(0)=V^{\perp}$ and $P^\ast_{\mu}(X^{\ast \ast})=\overline{V}^{w(X^{\ast \ast},X^\ast)}$.  Hence, consider $v^\ast_k \in Y$  such that $v^\ast_{k} \overset{w^\ast}{\to} v^\ast$ and $v^\ast = 0$ on $V$, so  $v^\ast = 0$ on $\overline{V}^{w(X^{\ast \ast},X^\ast)}$.  Moreover, because $v^\ast_k  \in Y$ and $P_\mu$ is a projection onto $Y$ one has $P_\mu(v_k^\ast)=v_k^\ast$, then $ \langle v^\ast , x - P_\mu(x)  \rangle =\lim \langle v_k^\ast , x- P_\mu^\ast(x) \rangle = \lim \langle P_\mu(v_k^\ast) , x- P_\mu^\ast(x) \rangle =\lim \langle  v_k^\ast , P_\mu^\ast(x)- P_\mu^\ast(x) \rangle =0$ for every $x \in X$, which implies (using that $ \langle v^\ast , P^\ast_\mu(x)  \rangle=0$) $ \langle v^\ast , x \rangle=0$.

	Now, we   choose  a  decreasing sequence of  positive numbers $\gamma_n\searrow 0^+$, consider $V_1 \times Y_1\in \mathcal{A}$  containing $(x', w^\ast)$, let $\{e(1,i)  \}_{i\in \N}$ be a dense set in $\mathbb{B}\cap V_1$ and define  
	$$W(1,p):=\{  y^\ast \in X^\ast : |\langle y^\ast ,  e(1,i) \rangle| \leq  \gamma_p, \text{ for all } i=1,...,p  \}.$$
	 Whence for all $p\ge 1$ and $t\in F$ we can pick  points $x_t(1,p) \in \mathbb{B}(x',f_t, \gamma_p)$, subgradients  $x^\ast_{t}(1,p) \in \hat{\sub} f_t(x_t(1,p))$, $\lambda (1,p) \in \Delta(F,x',\gamma_p)$ and $v(1,p)^\ast \in W(1,p)$ such that $w^\ast=\sum \lambda_{t}(1,p) x^\ast_{t}(1,p) + v^\ast(1,p)$. 
	
	Now assume that we have selected $V_n\times Y_n\in \mathcal{A}$ containing all $V_k\times Y_k$ for $k\leq n$, families of points $\{ e(n,i)  \}_{i\in \N}$  dense  in $\mathbb{B}\cap V_n$, which contains all  previous $\{ e(k,i)\}_{i \in \N}$ for $k \leq n$,  points $x_t(i,p) \in \mathbb{B}(x',f_t, \gamma_p)$, subgradients  $x^\ast_{t}(i,p) \in \hat{\sub} f_t(x_t(i,p))$, $\lambda(i,p) \in \Delta(F,x',\gamma_p)$ and $v(i,p)^\ast \in W(i,p)$ such that 
	\begin{align}\label{representationw}
w^\ast=\sum \lambda_{t}(i,p) x^\ast_{t}(i,p) + v^\ast(i,p), \text{ for } i\leq n \text{ and } p \geq 1.
		\end{align}
  Then, take $V_{n+1}\times Y_{n+1}\in \mathcal{A}$ such that $V_n \times Y_n \subseteq V_{n+1}\times Y_{n+1}$, $x_t(i,p) \in V_{n+1}$, $x^\ast_t(i,p) \in Y_{n+1}$ for all $t\in F$, $i \leq n$, $p \in \N$,  consider $\{ e(n+1,i) \}_{i\in \N}$ a dense set in $B\cap V_{n+1}$, and define 
  $$W(n+1,p):=\{  y^\ast \in X^\ast :| \langle y^\ast ,  e(k,i) \rangle| \leq  \gamma_p,\text{ for all } k=1,...,n+1 \text{ and } i=1,...,p\}.$$ Then for all $p\ge 1$ and $t\in F$ we can pick  points $x_t(n+1,p) \in \mathbb{B}(x',f_t, \gamma_p)$, subgradients  $x^\ast_{t}(n+1,p) \in \hat{\sub} f_t(x_t(n+1,p))$, $\lambda (n+1,p) \in \Delta(F,x',\gamma_p)$ and $v(n+1,p)^\ast \in W(n+1,p)$ such that $w^\ast=\sum \lambda_{t}(n+1,p) x^\ast_{t}(n+1,p) + v^\ast(n+1,p)$. 
	
	Now we  define  $\overline{ \bigcup_{n\in \N } V_n} \times\overline{ \bigcup_{n\in \N } Y_n} =:V\times Y \in \mathcal{A}$, $x_{t}(n):=x_{t}(n,n)$, $x^\ast_{t}(n):=x^\ast_{t}(n,n)$, $\lambda_{t}(n):=\lambda_{t}(n,n)$, $v^\ast(n):=v^\ast(n,n)$. Then, by our construction $ x_{t}(n) \overset{f}{\to} x'$. Since $\lambda(n) \in \Delta(F,x',\gamma_n)$  we can assume that $\lambda_t(n) \overset{n\to \infty}{\to} \lambda_t \in [0,1]$ for every $t\in F$, and  $\sum_{t\in F} \lambda_t =1$; moreover  $ f_t(x')=f_F (x')$  for   every $t \in \supp \lambda$.
	
	Then, on the one hand if (there exist some subsequence such that)   $\lambda_{t}(n) x^\ast_{t}(n) $ is bounded  for all $t\in F$, in this case  we can assume that 
	\begin{itemize}
		\item  If  $t\in \supp \lambda$, $\lambda_{t}(n) x^\ast_{t}(n)$ converge to some $\lambda_t x^\ast_t$ with   $x_t^\ast\in \sub f_t(x')$.
		\item  If  $t\notin \supp \lambda$, $\lambda_{t}(n) x^\ast_{t}(n)$ converge to some $ x^\ast_t\in \sub^\infty f_t(x')$.
		\item $v^\ast ({k} )\overset{w^\ast}{\to} v^\ast$.
	\end{itemize}
	 Furthermore,  $v^\ast$ is zero on $V$. Indeed, the set $\{ e(i,j)\}_{i,j}$ is dense in $V$, then  for every  $n\geq \max\{i,j\}$  we have that $| \langle v^\ast(n), e(i,j)| \leq \gamma_n$ (recall $v^\ast(n) \in W(n,n)$), so taking the limits $\langle v^\ast, e(i,j)\rangle=0$ for every $i,j$,  therefore $v^\ast$ is zero on $V$. Thus, by the property of $\mathcal{A}$ necessarily $v^\ast$ is zero on the whole $X$, hence using \cref{representationw}  we have that  \cref{eq:inclusion} holds.
	 
	  On the other hand, if there exists some $t\in F$ such that $\| \lambda_t(n) \cdot x_t^\ast (n)\|_{\ast} \to +\infty$, we  define $\eta_n:=(\max_{t\in F}\{ \| \lambda_{t}(n)x^\ast_{t}(n)\|_{\ast}, \| v^\ast (n)\|_{\ast} \})^{-1}$. We have $\eta_k w^\ast \to 0$ and (by passing to a subsequence) $\eta_n \lambda_{t}(n) x^\ast_{t}(n)\ \overset{w^\ast }{\to} w_t^\ast \in \sub^{\infty} f(x')$; and by a similar argument as in the first case $\eta_n v^\ast (n) \to 0$, so $\sum_{t\in F} w^\ast_t=0$. Moreover, by the limiting condition \cref{LIMITING:CONDITION} we have $w_{t}^\ast=0$. Finally, since all the functions but one of  $f_t$'s are SNEC at  $x'$ we have $\eta_n \lambda_{t}(n) x^\ast_{t}(n)$ converge in norm topology to zero, which is a contradiction. 
	
	Therefore $x^\ast \in \mathcal{S}(\bar{x},3\epsilon) + V + V$ ($x^\ast \in [0,\epsilon] \mathcal{S}(\bar{x},3\epsilon) + V + V$, resp.), and by the arbitrariness of $V$ and $\epsilon>0$ we conclude \cref{Mordukhovich:Separable01}.
\end{proof}

The next result gives us a simplification of the main formulae in \cref{Mordukhovich:Separable} under the additional assumption that  the data is  Lipschitz continuous. The case when the data is uniformly  Lipschitz continuous was proved in \cite[Theorem 3.2]{MR3033113}.

\begin{theorem}\label{TEO:MORD:NGH}
	Let  $\{ f_t :   t\in T \}$  be  a  family of  locally Lipschitz functions on a neighborhood of a point $\bar{x} \in \dom f$. 
	Then 
		\begin{align}\label{TEO:MORD:NGHi} 
	\partial f(\bar{x}) &\subseteq \bigcap\limits_{\epsilon >0} \cl^{w^\ast}  \bigg(     \mathcal{S}(\bar{x},\epsilon)\bigg), \text{ and }
	\partial^\infty f(\bar{x}) \subseteq \bigcap\limits_{\epsilon >0}  \cl^{w^\ast} \bigg( [0,\epsilon] \cdot   \mathcal{S}(\bar{x},\epsilon)\bigg),
	\end{align}
	where $\mathcal{S}(\bar{x},\epsilon)$ was defined in \cref{defn:Sxe}. In addition, if the family is uniformly locally Lipschitz at $\bar{x}$, then
	\begin{align}\label{TEO:MORD:NGHii} 
	\sub f(\bar{x}) \subseteq  \bigcap\limits_{\epsilon >0} \cl^{w^\ast}\left\{  \sum\limits_{t \in F } \lambda_t  \sub f_t(x') : \begin{array}{c}
	F\in \Pf(T_\epsilon(\bar{x})), x' \in \mathbb{B}(\bar{x},\epsilon), \\ \lambda\in \Delta(F) \text{ and } \\
	f_{t}(x')= f_F (x') \text{ for all } t\in F
	\end{array}  \right\}.
		\end{align}

\end{theorem}

	\begin{proof}
	Consider $V\in \mathcal{N}_{0}(w^*)$, $\epsilon >0$,   a finite-dimensional subspace $L\ni \bar{x}$ such that $L^\perp\subseteq V$ and $x^\ast \in \sub f(x)$ (respectively, $y^\ast \in \sub^{\infty} f(\bar{x})$ ), let $P:X \to L$ be a continuous linear projection and define $W=(P^\ast)^{-1}(V)$. Hence,  $x^\ast_{|_L} \in \sub f_{|_L}(x)$ (respectively,  $y^\ast_{|_L} \in \sub^\infty f_{|_L}(x)$).  Hence, we apply   \cref{Mordukhovich:Separable} and we conclude the existence of some  $F\in \mathcal{T}_{\epsilon}(\bar{x})$, $x' \in \mathbb{B}(\bar{x},\epsilon)$, $\lambda\in \Delta(F)$  such that $x^\ast_{|_L} \in \sum\limits_{t \in F } \lambda_t  \sub (f_{|_L})_t(x')  +W$ and $(f_{|_L})_{t'}(x')=(f_{|_L})_{t''}(x') \text{ for all } t',t^{''}\in  \supp \lambda$, then $$P^\ast (x^\ast_{|_L})\in  \sum\limits_{t \in F } \lambda_t  \sub (f_t+\delta_{L})(x')  +V=\sum\limits_{t \in F } \lambda_t  \sub f_t(x') + L^\perp  +V,$$ 
	where the last equality follows from the sum rule for Lipschitz functions (see \cite{MR2191744,MR859504,MR1014198}). Therefore $    x^\ast = P(x^\ast_{L}) + x^\ast - P(x^\ast_{L}) \in   \sum\limits_{t \in F } \lambda_t  \sub f_t(x') + V$, which implies $x^*\in \mathcal{S}(\bar{x},\epsilon) + V$. Similarly, for $y^\ast_{|_L} \in \sub^\infty f_{|_L}(x)$ one concludes that $y^\ast \in [0,\epsilon] \cdot   \mathcal{S}(\bar{x},\epsilon) + V$, and from the arbitrariness of $\epsilon>0$ and $V \in \mathcal{N}_0(w^\ast)$ we   conclude the proof of \cref{TEO:MORD:NGHi}.
	
	Finally to prove \cref{TEO:MORD:NGHii} we notice that if the functions are uniformly locally Lipschitz at $\bar{x}$ with constant $K$, then assuming that $\epsilon >0$ is small enough, we have that  for any $t\in T$, $x \in \mathbb{B}(\bar{x},\epsilon)$ and  $ | f_t(x)- f(\bar{x})| \leq \epsilon$  we also have $f_t(\bar{x}) \geq f(\bar{x})-(K+1) \epsilon$, which means $t\in T_{(K+1) \epsilon}(\bar{x})$.
\end{proof}
The next example shows an application of the above results with a family which is not uniformly locally Lipschitz. This example is important because, on the one hand  it provides an exact upper-estimation of the supremum function of a family of functions which are not uniformly locally Lipschitz, and, on the other hand it gives us a nonconvex upper-estimation.
\begin{example}\label{Example:Nonconvex}
	Consider $T=(0,1)$ and the family of functions $f_t:\mathbb{R}^2 \to \R $ given by 
	$$f_t(x,y)=tx^2 - \frac{|y| +1}{t}.$$ 
Here, it is important to notice that all the functions are Lipschitz continuous, but not uniformly Lipschitz continuous, so the results of \cite{MR3033113} cannot be applied. Nevertheless, we can apply  \cref{TEO:MORD:NGH}. Indeed, first the supremum function is given by  $f(x,y)=x^2 - |y| - 1$.  The limiting subdifferential of $f$ at $(\bar{x},\bar{y})=(0,0)$ is $\sub f(0,0)=\{ 0\} \times \{-1,1\}$ and the value of $f$ at this point is $f(0,0)=-1$. Now, we compute the limiting subdifferential of $f$ at $(\bar{x}, \bar{y})$ using \cref{TEO:MORD:NGH}. Pick $z^\ast$ in the right-hand side of  \cref{TEO:MORD:NGHi}, then there exist $\epsilon_n\to 0^+$, $F_n \in  \Pf(T)$,  $(x_{n},y_{n}) \in \epsilon_n\mathbb{B}, $ and $ \lambda_n \in \Delta(F_n)$  such that  $ | f_{t_n}(x_n,y_n)- f(0,0)| \leq \epsilon_n$,	$f_{t}(x_n,y_n)= f_{F_n} (x_n,y_n)$ for all $ t\in F_n$ and  $z^\ast_n \in \sum_{s\in F_n} \lambda_{s}  \sub f_{s} (x_n,y_n) +\epsilon_n \mathbb{B}^\ast$. Now the equation $$ t x_n^2 - \frac{|y_n| + 1}{t} = s x_n^2 - \frac{|y_n|+1}{s}$$ implies $t=s$, and consequently $F_n=\{t_n\}$. 
	
	Now,using the inequality 	$| f_{t_n}(x_n,y_n)- f(0,0)|=|f_{t_n}(x_n,y_n) +1|\leq \epsilon_n$ one gets $t_n \to 1$. Therefore, 	$z^\ast_n\in \{ (2t_n x_n^2, \frac{1}{t_n}), (2 t_n x_n^2 , -\frac{1}{t_n}) \}+\epsilon_n \mathbb{B}^\ast$ with $t_n \to 1$, $x_n \to 0$ and $\epsilon \to 0$,  consequently $z^\ast \in \{ 0\} \times \{-1,1\}$.
\end{example}

{In order to derive a more  precise estimation of the subdifferential of the supremum function in \cite[Definition 3.4]{MR3033113}, the authors introduced the definition  of \emph{equicontinuous
		subdifferentiablitity}. This notion involves some \emph{uniform
		continuity}  of the subdifferentials of the data functions $f_t$'s for
	points close to the active index set. }

\begin{definition}
	\label{def:equicontinuity:2} Let $f_t : X \to \mathbb{R} \cup \{\infty\}$
	be a family of lsc functions indexed by $t \in T$. The family
	is called  equicontinuously subdifferentiable at $\bar x \in X$ if
	for any weak$^\ast$-neighborhood $V$ of the origin in $X^*$ there is
	some $\epsilon > 0$ such that
	\begin{equation}\label{equicontinuity:equation}
	\sub f_t(x) \subseteq \sub f_t(\bar x) + V, \text{ for all } t \in T_{\epsilon}(\bar{ x}) \text{ and all } x \in \mathbb{B}(\bar x,\epsilon).
	\end{equation}
\end{definition}
{Although this definition is precisely for the framework of \cite{MR3033113}, our formulae involves the singular subdifferential of the nominal data for points close to the point of interest, due to the possible lack of Lipschitz continuity of our data. For that reason we introduce the following definition, which is satisfied trivially when the nominal data is Lipschitz continuous in a neighborhood of the point of interest.}
\begin{definition}
	Let $f_t : X \to \mathbb{R} \cup \{\infty\}$
	be a family of lsc functions indexed by $t \in T$. The family
	is called singular equicontinuously subdifferentiable at $\bar x \in X$ if 	for any weak$^\ast$-neighborhood $V$ of the origin in $X^*$ there is
	some $\epsilon > 0$ such that
		\begin{equation}\label{equicontinuity:equation0}
	\sub^\infty f_t(x) \subseteq \sub^\infty f_t(\bar x) + V, \text{ for all } t \in T \text{ and all } x \in \mathbb{B}(\bar x,\epsilon).
	\end{equation}
	\end{definition}
Finally, we say that the family of functions $\{ f_t: t\in T \}$ is \emph{total  equicontinuously subdifferentiable} at $\bar x \in X$ if $\{ f_t: t\in T \}$ is equicontinuously subdifferentiable  and singular equicontinuously subdifferentiable at $\bar x \in X$ .

{
Using the  notion of \emph{total  equicontinuously subdifferentiable}  we have the following tighter formulae, which represents an extension of \cite[Proposition 3.5]{MR3033113}.}
{
\begin{theorem}\label{Theorem:equicontinuouslysubdifferentiable}
	In the setting of \cref{Mordukhovich:Separable} assume that the family of functions $\{f_t \}_{t\in T}$ is total equicontinuously subdifferentiable at $\bar x$ and 
	\begin{align}\label{uniformcontinuity}
	\lim\limits_{x\to \bar{x}}\sup\limits_{t\in T} |f_t(x)-f_t(\bar{x})| =0.
	\end{align}
	Then
	\begin{align}
\partial f(\bar{x}) &\subseteq \bigcap\limits_{\epsilon >0} \cl^{w^\ast}  \left\{  \sum\limits_{t \in T } \lambda_t \circ \sub f_t(\bar{x}) :\begin{array}{c}
	 \lambda \in  \Delta(T) \text{ and }\\
	  \supp  \lambda \subseteq  T_\epsilon(\bar{x}) 
\end{array}
 \right\}  \text{ and } \label{Theorem:equicontinuouslysubdifferentiable01}\\
\partial^\infty f(\bar{x}) &\subseteq \bigcap\limits_{\epsilon >0}  \cl^{w^\ast} \bigg( [0,\epsilon] \cdot   \left\{  \sum\limits_{t \in T } \lambda_t \circ \sub f_t(\bar{x}) :  \begin{array}{c}
	\lambda \in  \Delta(T) \text{ and }\\
	\supp  \lambda \subseteq  T_\epsilon(\bar{x}) 
\end{array}  \right\}\bigg).\label{Theorem:equicontinuouslysubdifferentiable02}
\end{align}
\end{theorem}}
\begin{proof}
	Consider $x^\ast \in \sub f(\bar{x})$, $\epsilon >0$ and $V $ a weak$^\ast$-neighborhood  of the origin. First, by \cref{equicontinuity:equation,equicontinuity:equation0}  we can take $\gamma_1>0$ such that   for all $ x \in \mathbb{B}(\bar x,\gamma_1)$
	\begin{align}
	\sub f_t(x) \subseteq &\sub f_t(\bar x) + V, \text{ for all } t \in T_{\gamma_1}(\bar{ x}) \text{ and }\label{eq:001} \\
		\sub^\infty f_t(x) \subseteq &\sub^\infty f_t(\bar x) + V \text{ for all } t \in T.\label{eq:002}
	\end{align}
Second,	by \cref{uniformcontinuity} we can take $\gamma_2>0 $ such that 
	\begin{align}\label{eq:01}
		|f_t(x) - f_t(\bar{x})| \leq \gamma_1/2, \;\forall  t\in T, \;  \forall x\in \mathbb{B}(\bar{x},\gamma_2).
	\end{align}
	 Now, by \cref{Mordukhovich:Separable} we have that for  $\gamma =\min\{ \gamma_1/2,\gamma_2,\epsilon/2 \}$ 
	 $$  x^\ast \in  \mathcal{S}(\bar{x},\gamma) + V.$$
Whence, there exists $F\in \Pf(T)$, $\lambda \in \Delta(F)$  and $x'\in \mathbb{B}(\bar{x},\gamma)$ such that  $|f_F(x')-f(\bar{x})| \leq \gamma$ and  $f_F(x')=f_t(x')$ for all $t\in \supp  \lambda$ and 
\begin{align}\label{eq:003}
	x^\ast \in \sum_{t\in F} \lambda_t \circ \sub f_t(x') +V.    
	\end{align}
 Hence, by \cref{eq:01}  we have that for all $t\in \supp \lambda$
\begin{align*}
f(\bar{x})  &\leq f_F(x') + \gamma  = f_t(x') +\gamma \leq f_t(\bar{x}) +\gamma_1/2 + \gamma\leq f_t(\bar{x} ) +\gamma_1,
	\end{align*}
which means that $t \in T_{ \gamma_1 }(\bar{x})$ and consequently $\supp \lambda   \subseteq  T_{\gamma_1} (\bar{x})$. Now, by  \cref{eq:001,eq:002,eq:003} we have
\begin{align*}
	x^\ast \in &\sum_{\lambda_t >0 } \lambda_t \cdot \sub f_t(x') + \sum_{\lambda_t =0 }  \sub^\infty f_t(x')  +V \\
	\subseteq  & \sum_{\lambda_t >0 } \lambda_t \circ \sub f_t(\bar{ x}) +\sum_{\lambda_t =0 }  \sub^\infty f_t(\bar{ x})+ V + V + V \\
	\subseteq &   \left\{  \sum\limits_{t \in T } \lambda_t \circ \sub f_t(\bar{x}) :\begin{array}{c}
		\lambda \in  \Delta(T) \text{ and }\\
		\supp  \lambda \subseteq  T_\epsilon(\bar{x}) 
	\end{array}
	\right\} + V + V + V.
	\end{align*}
Finally, from the arbitrariness of $\epsilon$ and $V$ we conclude \cref{Theorem:equicontinuouslysubdifferentiable01}. The proof of \cref{Theorem:equicontinuouslysubdifferentiable02} is similar, so we omit the proof.
	\end{proof}

\section{The convex subdifferential}\label{SECTION:CONVEXSUB}
This section is devoted to giving formulae for the convex subdifferential. Due to the closure of the graph of the convex  subdifferential  under bounded nets with respect to the $\| \cdot \| \times w^\ast$-topology in $X\times X^\ast$, we can obtain  a similar result to    \cref{Mordukhovich:Separable}  by changing the SNEC  assumption for a similar one using  nets instead of sequences.  For this purpose, it is better to  express the limiting condition of   \cref{Mordukhovich:Separable}   in terms of the  normal cone of the domain of each function $f_t$, more precisely, we recall that for any lsc convex function $h$,  the normal cone to the domain of $h$ at a point $x$ is given by   $$ N_{\dom h}(x):=\{ x^\ast \in X^\ast : \langle x^\ast , y-x\rangle \leq 0, \; \forall y \in \dom h \}.$$ Using this notation we establish the following result.

\begin{theorem}\label{TEO:CONVEX:1}
	Let  $\{ f_t :   t\in T \}$ be a family of proper convex lsc functions satisfying the following assumptions:  There exists a neighborhood $U$ of $\bar{x}$ such that
	\begin{enumerate}[label=\alph*)]
		\item For  all $x\in U$, all but one of the functions $\{ f_t :   t\in T \}$ and every net  $(\lambda_\nu, x_\nu, x_\nu^\ast) \in [0,+\infty) \times X\times X^\ast$ satisfying $\lambda_\nu \to 0$, $x_\nu \overset{f}{\to} x$, $x^\ast_\nu\in \sub f(x_\nu)$ and $\lambda_\nu x^\ast_\nu \overset{*}{\rightarrow} 0$ one has $\| \lambda_\nu x^\ast_\nu\|_{*} \to 0$.
		\item  For  all $x\in U$ and all  $F\in \Pf(T)$ 
		\begin{align*}
			w^\ast_t \in N_{\dom f_t }(x), \; t\in F \text{ and } \sum_{t\in F} w^\ast_t =0 \text{ implies  } w^\ast_t =0, \text{ for all } t \in F.
		\end{align*}
		Then 
		\begin{align}\label{EQUATION:CONVEX1}
			\sub f(\bar{x}) \subseteq \bigcap\limits_{\epsilon >0} \cl^{w^\ast}  A(\bar{x},\epsilon).
		\end{align}
	Where
		$$A(\bar{x},\epsilon):=\bigcup \bigg\{ \co\bigg( \bigcup_{_{t\in  F_1} }\sub f_t(x') \bigg) +  \sum\limits_{t \in F_2 } N_{\dom f_t}(x') \bigg\}$$
		and the union is over all $	F_1, F_2\in \Pf(T)$ and $x' \in \mathbb{B}(\bar{x},\epsilon)$ such that  $|f_t(x')- f(\bar{x})| \leq \epsilon$  and $f_{t}(x')=f_{F_1\cup F_2}  (x') \text{ for all } t' \in F_1$.
		Moreover, the equality holds, whenever the function $f$ is continuous at some point, or the space $X$ is finite-dimensional.
	\end{enumerate}
\end{theorem}
\begin{proof}
	Since the proof of \cref{EQUATION:CONVEX1} relies on similar arguments as   \cref{Mordukhovich:Separable} (but without the use of techniques of separable reduction) we prefer to omit the proof. Now, any point in the right-hand side of \cref{EQUATION:CONVEX1} is the limit of a net  $w^\ast$, which has the form of $w^\ast_\nu=\sum \lambda_\nu(t) v_\nu(t)^\ast + \sum w^\ast_\nu(t)$ with $v_\nu \in \sub f_t(x_\nu)$, $w_\nu^\ast(t)\in N_{\dom f_t}(x_\nu)$, $\sum \lambda_\nu(t) =1$ and  $\lambda_\nu(t)\geq 0$, then one gets for every $y\in  X$ 
	\begin{align*}
		\langle w^\ast_\nu , y - \bar{x} \rangle \leq f(y) -f(x) +  | f_t(x')- f(\bar{x})|  + \langle w^\ast_\nu ,   x_\nu- \bar{x} \rangle.
	\end{align*}
	Therefore, we can conclude the  equality in \cref{EQUATION:CONVEX1} whenever the $\lim \langle w^\ast_\nu ,   x_\nu- \bar{x} \rangle=0$, and this holds in particular when the function $f$ is continuous at some point, or the space $X$ is finite-dimensional, because in these  cases the net $\{  w^\ast_\nu\}$ is bounded.
\end{proof}
The following  results have the intention of establishing formulae without any qualification. This is possible by reducing  the analysis to subspaces with nice properties for the family of functions. For that reason we denoted by  $\mathcal{F}_x$ the set of all finite-dimensional affine subspaces containing $x$. This class of sets allows us to give  formulae in any (Hausdorff) locally convex  topological vector space (lcs for short).  It is useful  to recall some simple facts about lcs available in  pioneer  books such as \cite{MR910295,MR1741419}: The topology on every lcs $X$ is generated by a family of seminorms $\{ \rho_i : i \in \mathcal{I}\}$, which will be always assumed to be \emph{up-directed}, i.e., for every two points $i_1,i_2\in \mathcal{I}$ there exists $i_3 \in \mathcal{I}$ such that $\rho_{i_3} (x) \geq \max\{\rho_{i_1}(x),\rho_{i_2}(x) \}$ for all $x\in X$. For  a point $\bar{x}$ in $X$, $r\geq 0$ and a seminorm $\rho$ we define $\mathbb{B}_{\rho}(\bar{x},r):= \{ x \in X: \rho(x-\bar{x}) \leq 0 \}$. In the (topological) dual of   $X$, denoted by $X^\ast$, some examples of topologies are  the \emph{$w^{\ast }$-topology} denoted by $w(X^{\ast },X)$ ($w^{\ast },$ for	short), which is the topology generated by the pointwise convergence, and the \emph{strong topology} denoted by $\beta (X^{\ast },X)$  ($\beta$, for short), which is the topology generated by the uniform convergence on bounded sets. For a set $A\subseteq X^\ast$, the symbol $\beta\text{-}seq\text{-}A$ denotes the set of points which are the limit, with respect to the $\beta$-topology, of some sequence lying in $A$. Finally, for a function $g:X\to \Rex$, $\cco g$ denotes the convex lsc envelope of $g$. For more details  about the theory of convex analysis in lcs we refer to \cite{MR0467080,MR1921556,moreau1967fonctionnelles}.

Now, let us establish the first general formula without any qualification condition.
\begin{theorem}\label{corollaryconvex1}
	Let $X$ be an lcs, let $\mathcal{I}$ be a family of seminorms which generate the topology on $X$. Consider a   family of proper convex lsc functions  $\{ f_t :   t\in T \}$. Then, for all $\bar{x}\in X$ 
	\begin{align}\label{corollaryconvex1:eq1}
		\sub f(\bar{x})=\bigcap\limits_{\substack{\epsilon >0, \rho \in \mathcal{I}\\ L\in \mathcal{F}_x }} \beta\text{-}seq\text{-}\cl A_{\epsilon,L,\rho}(\bar{x}),
		\end{align}
	where
	\begin{align*}
	A_{\epsilon,L,\rho}(\bar{x}):= \bigcup \bigg\{  \co\bigg( \bigcup_{t\in  F_1}\sub  f_{t,L}(x') \bigg) +  \sum\limits_{t \in F_2 } N_{\dom f_t \cap L}(x')  \bigg\}    .
	\end{align*}
Where $f_{t,L}:=f_t + \delta_{\aff(\dom f \cap L)}$ and the union is over all $x' \in \mathbb{B}_{\rho}(\bar{x},\epsilon)\cap L$ and $F_1, F_2\in \Pf(T)$ such that $	f_{t}(x')=f_{F_1 \cup F_2} (x')$  for all  $t \in F_1$ and $|f_t(x')- f(\bar{x})| \leq \epsilon$.
\end{theorem}
\begin{proof}
	W.l.o.g. we may assume that $\bar{x}=0$. Consider $\epsilon>0$, $L\in   \mathcal{F}_x$, and $\rho$ a seminorm on $X$, also we can assume that $\rho$ is a norm on $L$, because $A_{\epsilon,L,\rho_1}(0) \subseteq A_{\epsilon,L,\rho}(0)$,  for any $\rho_1 	\geq \rho$. Consider $W:=\aff(\dom f \cap L)$, let us show that
	 \begin{equation}\label{eq:100}
		\sub (f + \delta_{W})(0)\subseteq \beta\text{-}seq\text{-}\cl  A_{\epsilon, L,\rho }(0).
	\end{equation}	 
Indeed, take  $x^\ast \in \sub (f + \delta_{W})(0)$ and let  $P : X\to (W,\rho)$ be a continuous linear projection. Hence, $x^\ast_{|_W}$ (the restriction of  $x^\ast$ to $W$) belongs to $\sub f_{|_W}(0)$. The finite-dimensionality of $W$ gives us the continuity of $f_{|_W}$ at some point (see \cite{MR1451876}), so    the family $(f_t)_{|_W}$ satisfies the hypotheses of  \cref{TEO:CONVEX:1}. Whence, there exists a sequence $w^\ast_n \to x^\ast_{|_W}$ 
  where 
  $$w^\ast_n \in \co\bigg( \bigcup_{_{t\in  F_{1,n}} }\sub ((f_t)_{|W})(x_n') \bigg)  +  \sum\limits_{t \in F_{2,n} } N_{\dom (f_t)_{|W}}(x_n') $$
with $F_{1,n},F_{2,n} \in \Pf(T)$, $x_n' \in \mathbb{B}_{\rho}(0,\epsilon)\cap W$ such that  $|f_t(x_n')- f(\bar{x})| \leq \epsilon$  and $f_{t}(x_n')=\max_{F_{1,n}\cup F_{2,n}}  f_t (x')$  for all  $t \in F_{1,n}$.

	Now we define $x_n^\ast:=P^\ast(w_n^\ast)+ x^\ast - P^\ast(x^\ast_{|W})$, it follows that 
$x_n^\ast\in 	A_{\epsilon,L,\rho}(0)$. Moreover, considering $V:=P^{-1}(\mathbb{B}_{W})$, where $\mathbb{B}_{W}$ is the unit ball in $W$, we get 
\begin{align*}
	\sigma_{V}(x^\ast - y_n^\ast)&=\sup_{v\in V} \langle x^\ast - y^\ast_n, v\rangle = \sup_{v\in V} \langle P^\ast(w_n^\ast) - P^\ast(x^\ast_{|W}) , v\rangle  \\&= \sup_{h \in \mathbb{B}_{W}} \langle z_n^\ast- x^\ast_{|W} , h\rangle \to 0.
	\end{align*}
Which concludes  \cref{eq:100}, then using that  $$\sub f(0)=\bigcap\limits_{L\in \mathcal{F}_0} \sub (f +\delta_{\aff (\dom f \cap L)})(0) \subseteq \bigcap\limits_{\substack{\epsilon >0, \rho \in \mathcal{I}\\ L\in \mathcal{F}_0 }} \beta\text{-}seq\text{-}\cl A_{\epsilon,L,\rho}(\bar{0}),$$ we get the first inclusion  in \cref{corollaryconvex1:eq1}.

	Now, pick $x^\ast \in \bigcap\limits_{\substack{\epsilon >0, \rho \in \mathcal{I}\\ L\in \mathcal{F}_0 }} \beta\text{-}seq\text{-}\cl A_{\epsilon,L,\rho}(0)$ and  $y \in  \dom f$. Then, take a sequence $\epsilon_n \to 0$ and pick $L\in \mathcal{F}_0$ which contains $y$ and  consider $\rho \in \mathcal{I}$ such that $\rho$ is a norm on $L$ and  $\rho(x_n) \to 0$ implies $|\langle x^\ast, x \rangle |\to 0$. Hence, there exist  sequences $F_{1,n},F_{2,n} \in \Pf(T)$, $x_n \in \mathbb{B}_{\rho}(0,\epsilon_n)\cap L$ and  $w^\ast_{n} \in X^\ast$ such that  $w^\ast_n \overset{\beta}\to x^\ast$,   
	\begin{align*}
		w^\ast_n  \in \co\bigg( \bigcup_{t\in  F_{1,n}}\sub  f_{t,L}(x_n) \bigg) +  \sum\limits_{t \in F_{2,n} } N_{\dom f_t \cap L}(x_n) 
		\end{align*}
and $|f_t(x_n)- f(0)| \leq \epsilon_n$,   $f_{t}(x_n)=\max_{F_{1,n}\cup F_{2,n}}  f_t (x_n)$  for all  $t \in F_{1,n}$, which implies
		\begin{align}\label{eq234}
	\langle w^\ast_n, y - x_n \rangle \leq f(y) - f(0) + \epsilon_n.
	\end{align}

We claim that $	\langle w^\ast_n, y - x_n \rangle \to 	\langle x^\ast, y \rangle$. Indeed, because $\rho$ is a norm in $L$, $x_n \in L$ and $\rho(x_n) \to 0$ necessarily $x_n \to 0$ with respect to the topology on $X$. Hence, the set $B:=\{y- x_n: n\in \N\} $ is bounded, so 
\begin{align*}
|	\langle w^\ast_n , y-x_n \rangle  -\langle x^\ast, y \rangle | &=  |	\langle w^\ast_n -x^*, y-x_n \rangle  -\langle x^\ast, x_n  \rangle  |\\&\leq  \sigma_{B}(w^\ast_n -x^*) + 	|\langle x^* , x_n \rangle|\to 0.
	\end{align*}
Finally, taking $n \to \infty$ in \cref{eq234}  it yields 
$\langle x^\ast , y - x \rangle \leq f(y) - f(0)$, which concludes the proof due to the arbitrariness of $y\in \dom f$.
\end{proof}

The final goal of this paper is to give an alternative proof of \cite[Corollary 6]{MR3561780}, which, as far as we know, appears to be the most general extension of \cite[Theorem 4]{MR2448918}. Before presenting this proof we need the following lemma. This result  is interesting by itself, since it allows us to understand the subdifferential of any function in terms of the subdifferential of another function.
\begin{lemma}\label{lemma:Rock}
	Let $X$ be an lcs, let $h,g: X \to \Rex$ be two convex lsc proper functions and let $D \subseteq \dom h$ be a convex subset such that 
	$$h (x)=g(x) \text{ for all }x\in D.$$
	 Then for every $\bar{x}\in X$ 
	\begin{align}\label{We2}
	\sub (h+\delta_{D}) (\bar{x}) =  \bigcap\limits_{ \substack{ L\in \mathcal{F}_{\bar{x}}}} \bigg\{\co\left\{ S_L(\bar{x}) \right\}+ N_{ D \cap  L}(\bar{x})\bigg\},
\end{align}
	where	$S_L(\bar{x}):=\limsup \sub (g+ \delta_{\aff(D\cap L)})(x')$, the $\limsup$  is understood to be the set of all  $x^\ast \in X^\ast$, which are the  limit (in the $\beta$-topology) of some sequence $x^\ast_n \in \sub (g+ \delta_{\aff(D\cap L)}(x_n)$ with $x_n \in \ri_{L}(D)$,   $x_n \overset{g}{\to}\bar{x}$ and $|\langle x_n^*, x_n - \bar{x}\rangle| \to 0$. Here, $ \ri_{L}(D)$ denotes  the interior of $D\cap L$ with respect to $\aff(D\cap L)$.
\end{lemma}
\begin{proof}
	W.l.o.g. we may assume that $\bar{x}=0$. First we notice  that
		\begin{align}\label{We0}
		\sub (h + \delta_{D})(0) = \bigcap_{ L\in \mathcal{L}_0 }  \sub (h + \delta_{D\cap L} )(0)= \bigcap_{ L\in \mathcal{L}_0 }  \sub (h + \delta_{\cl(D\cap L)} )(0).
	\end{align} 
	Indeed, the first inequality is straightforward and the second follows from the fact that $\sub (h + \delta_{D\cap L} )(0) =  \sub (h + \delta_{\cl(D\cap L)} )(0)$ thanks to the \emph{accessibility lemma} (see, e.g., \cite{MR1992991}).	
	Now, fix   $ L\in \mathcal{F}_0$, define $W= \aff ( L \cap D )$ and consider a  continuous linear projection $P:X \to W$. We claim that 
	\begin{align}\label{We}
		\sub (h + \delta_{\cl(D\cap L) })(0) \subseteq \co\left\{ S_L(0) \right\}+ N_{\dom f \cap D \cap  L}(0).
		\end{align}
	Indeed, take $x^* \in \sub (h + \delta_{\cl( L\cap  D) })(0)$, using the same finite-dimensional representation as in the proof of   \cref{corollaryconvex1}, one gets the existence of a point $y^\ast  \in \sub (h+\delta_{D\cap L} )_{|_W}(0)$ and $z^\ast \in W^{\perp}$ such that $x^\ast = P^\ast(y^\ast) + z^\ast$. Then, by the finite-dimensionality of $W$ $\ri_{\aff(D\cap L)}$ is not empty and consequently $(h + \delta_{\cl( L\cap  D) })_{|_W}$ has a point of continuity (relative to its domain). Then, we  apply \cite[Theorem 25.6]{MR1451876} and  we get the existence of  sequences $u_{n,i} \in \ri(\dom (h + \delta_{\cl( L\cap  D) })_{|_W})$, $ y_n^\ast, u_{n,i}^\ast \in W^\ast$,  $ \alpha_{n,i} \geq 0$ with $\sum_{i=1}^N \alpha_i=1$ and a point $\theta^\ast \in N_{\dom (h_{|_W})}(0)$  such that $y^\ast = \lim y_n^\ast  + \theta^\ast$, $y_n^\ast=\sum_{i=1}^{N} \alpha_{n,i} u^\ast_{n,i}$,   $u_{n,i}^\ast\in \sub (h + \delta_{\cl( L\cap  D) })_{|_W}(u_{n,i})$   and $u_{n,i} \to 0$, where the number $N= \dim W +1$ is fixed   by virtue of  Carath\'eodory's Theorem.
	
	 Now, $\sub (h + \delta_{\cl( L\cap  D) })_{|_W}(u_{n,i})= \sub h_{|_W}(u_{n,i})$, because $u_{n,i} \in \ri_{L}(D)$. Furthermore, $h(x') =g(x')$ for every $x'\in \ri_{L}(D)$, which implies that $u_{n,i}^\ast\in \sub g_{|_W}(u_{n,i})$.
	
	Moreover, the vectors  $ \alpha_{n,i} u^\ast_{n,i}$ must be bounded (to prove  this fact, one can argue by contradiction following the proof of  \cref{Mordukhovich:Separable}, and then one shows that $ N_{\dom (h+\delta_{\cl D\cap L})_{|_W} }(0)$ contains  a line, which is not possible due to the continuity of $(h+\delta_{\cl D\cap L})_{|_W} $). Hence, we may assume that $ \alpha_{n,i} u^\ast_{n,i}$ converges and  $\alpha_{n,i} \overset{n\to \infty}{\longrightarrow} \alpha_i $. More precisely, on the one hand for each index $i$ such that $\alpha_i=0$, one has that  $\alpha_{n,i} u^\ast_{n,i} \to v_i^\ast$  and $v_i^\ast\in N_{\dom f_{|_W} }(0)$. Indeed, for every $y\in \dom  h_{|_W}$
	\begin{align*}
		\langle  	v_i^\ast , y - 0 \rangle &= \lim 	\langle \alpha_{n,i} u^\ast_{n,i}, y -u_{n,i} \rangle + 	\lim\langle \alpha_{n,i} u^\ast_{n,i}, u_{n,i} -0\rangle\\
		&\leq \lim \alpha_{n,i}(h(y)- h(u_{n,i})  + \lim\langle \alpha_{n,i} u^\ast_{n,i}, u_{n,i} -0\rangle=0.
	\end{align*}
On the other hand, we have that for every index $i$ such that $\alpha_i\neq 0$, $u^\ast_{n,i} \to v_i^*$ and $| \langle u^\ast_{n,i}, u_{n,i}\rangle| \to 0$, then using that $u_{i,n}^\ast \in \sub g_{|_W}(u_{i,n})$ we get  $g(u_{n,i}) \to g(0)$.  	Therefore, 
\begin{align*}
	y^\ast=\sum\limits_{\{ i \mid \alpha_i\neq 0  \}} \alpha_i v_i^\ast + \sum\limits_{ \{  i \mid \alpha_i= 0 \} }  v_i^\ast + \theta^\ast,
	\end{align*}
	 with  $v_i^\ast \in \limsup \sub  f_{|_W}(u_{n,i})$  and  $q^\ast:=\sum\limits_{ 	\{ i \mid \alpha_i= 0 \} }  v_i^\ast + \theta^\ast \in N_{\dom f_{|_W} }(0)$.
	
	Now  define $w_{i}^\ast:=P^\ast(v^\ast_i)$, $\lambda^\ast:=z^\ast + P^\ast(q^\ast)$, $w^\ast:= \sum\limits_{\{ i \mid \alpha_i\neq 0  \}} \alpha_i w_i^\ast$, $w_{n,i}=P^\ast(u^\ast_{n,i})$,  it follows that $w_{n,i}^\ast \overset{ \beta }{\longrightarrow } w_i^\ast$, $| \langle w_{n,i}^\ast, u_{n,i}| \to 0$ and  $w^\ast_{n,i} \in \sub (g + \delta_{W})(u_{n,i})$, $u_{n,i} \in \ri_L(\dom h)$, $g(u_{n,i} )\to g(0)$, $\lambda^\ast \in  N_{\dom h \cap L}(0)$ and $x^\ast = w^\ast + \lambda^\ast$, which concludes the proof of \cref{We}. Then, using \cref{We0,We} we conclude  the first inclusion in  \cref{We2}.
	
	To prove the opposite  inclusion, consider $x^\ast $ in the right-hand side of \cref{We2} and $y\in D$, and consider $L$ as the subspace generated by $y$. Then, there are $\alpha_i \geq 0$ (with $\sum_{i} \alpha_i =1$), $x^\ast_{n,i} \in \sub (g+ \delta_{\aff(D\cap L)}(x_{n,i})$ and  $x_{n,i} \in \ri_{L}(D)$ such that   $x_{n,i} \overset{g}{\to}0$, $x^\ast_{n ,i} \overset{\beta}{\to} y_i ^*$, $|\langle x_{n,i}^*, x_{n,i} \rangle| \to 0$  and $x^\ast = \sum_{i} \alpha_i y_i^\ast + \lambda^\ast$ . Moreover, because $x_n \in \ri_L(D)$ and $h=g$ in $D$, we get  $\sub (g+ \delta_{\aff(D\cap L)})(x_n)= \sub (h+ \delta_{\aff(D\cap L)})(x_n)$. Then, 
	\begin{align*}
		\langle  x^\ast ,y  \rangle &= \langle \sum_{i} \alpha_i y_i^\ast + \lambda^\ast  ,y  \rangle\leq  \sum_{i} \alpha_i   \lim_{n} \langle  x_{n,i}^\ast , y -x_{n,i} \rangle +  \lim_{n} \langle  x_{n,i}^\ast , x_{n,i} \rangle\\ &\leq  \sum_{i} \alpha_i   \lim_{n}  ( h(y)- h(x_{n,i})) = h(y)-h(0).
		\end{align*}
	From the arbitrariness of $y$ we conclude that $x^*\in \sub (h +\delta_D)(0)$, which concludes the proof of \cref{We2}.
\end{proof}

\begin{theorem}\label{teofinal}
	Let $X$ be an lcs and let $\{ f_t :   t\in T \}$ be an arbitrary family of functions and let $D \subset \dom \cco f$ be a convex set such that
	$$\cco (f+\delta_{D})(x) =\sup_{t\in T} \cco f_t(x)  \text{ for all  } x\in D.$$
	   Then for all $\bar{x}\in X$
	\begin{align}\label{FORMULA:CONVEX2}
		\sub (f+\delta_D)(\bar{x})=\bigcap\limits_{\substack{\epsilon>0\\ L\in \mathcal{F}_{\bar{x}} }}  \cl^{w^*}\bigg(  \co\big(\bigcup\limits_{ t  \in  T_\epsilon(\bar{x}) }  \sub_{\epsilon}f_t(\bar{x}) \big)+ N_{D \cap L}(\bar{x})\bigg).
	\end{align}
\end{theorem}
\begin{proof}
W.l.o.g we can assume that $\bar{x}=0$. Because the inclusion $\supseteq$ is direct, we focus on the opposite one. To prove  this inclusion, we can assume that $\sub (f+\delta_D)(0) \neq \emptyset$, in particular $(f+\delta_D)(x)=\cco(f+\delta_D)(x)$.   First, we denote by $h=\cco (f+\delta_{D})$,  $g_t:=\cco f_t$ and $g=\sup_{t \in T} g_t$, then we apply  \cref{lemma:Rock} and we get
	\begin{align}\label{zero}
	\sub (f+\delta_{D}) (0)\subseteq 	\sub h (0) =  \bigcap\limits_{ \substack{ L\in \mathcal{F}_{0}}} \bigg\{\co\left\{ S_L(0) \right\}+ N_{ D \cap  L}(0)\bigg\}.
	\end{align}
	
	 We claim that for every  $L \in \mathcal{F}_0$, $\epsilon>0$ and $U\in\mathcal{N}_0(w^*)$
	\begin{align}\label{primero}
			S_L(0) \subseteq     \co \big(  \bigcup\limits_{ t  \in  T_\epsilon(0) }  \sub_{\epsilon}f_t(0) \big)+ N_{D \cap L}(0) + U+U,
		\end{align}
	where $S_L(0)$ was defined in   \cref{lemma:Rock}. Indeed, consider $x^\ast \in S_L(0)$, then by definition there exist sequences $y_n\in \ri_{\aff(D \cap L)} (D)$ and  $y_n^\ast \in \sub (g+ \delta_{\aff(D\cap L)})(y_n)$ such that $y_n^\ast  \to x^\ast$, $|\langle y^\ast_n, y_n\rangle | \to  0$ and $| g(y_n) -g(0) | \to 0$.

	 Now, the restriction of each $y_n^\ast$ to $W:=\aff(D\cap L)$ belongs to $\sub g_{|_W}(y_n)$ and $y_n \in \ri_W(\dom g_{|_W})$.  Since the function $g_{|_W}$ is locally bounded at $y_n$ we can find a constant $M_n$  and a closed convex neighborhood $V_n$ of zero (relative to $W$) such that
	 \begin{align*}
	 	g_t(x) \leq g_t(y_n) + M_n-g_t(y_n),\; \forall x \in y_n+ V_n.
	 	\end{align*}
 	Consequently,  by   \cite[Corollary 2.2.12]{MR1921556}
 	\begin{align*}
 		|g_t(x)-g(x')| \leq 3M_{t,n} \rho_{V_n}(x-y),\forall x,x' \in y_n + \frac{1}{2}V_n,
 		\end{align*}
 	where $M_{t,n}:=M_n-g_t(y)$ and $\rho_{V_n}$ is the \emph{Minkowski's functional} associated to $V_n$, that is, $\rho_{V_n}(u):=\inf\{ s >0 : u \in sV_n	\}$. In particular,  each function $(g_t)_{|_W}$  is   Lipschitz continuous on   $\frac{1}{2}V_n$, it allows us to apply \cref{TEO:MORD:NGH} and by a diagonal argument   we yield  that there exists a sequence  of sets 
 	  $F_n \in \Pf(T)$, and  there are  sequences of vectors $x_n \in W $ , $x^\ast_t(n) \in \sub (g_t)_{|_W}(x_n)$ together with scalars $(\lambda_t(n)) \in \Delta(F_n)$ such that $x_n \to 0$, $|g_{F_n}(x_n)-g(0)| \to 0$
 	  and $g_t(x_n)=g_{F_n}(x_n)$ for all $t\in F_n$ and $x_n^\ast= \sum_{t\in F_n} \lambda_t(n) x^\ast_t(n) \to x_{|_W}^\ast$.  From the fact that the dimension of $W$ is finite, we can assume that $\#F_n \leq \dim(W)+1$. Hence, necessarily the points $x^\ast_t(n)$ are uniformly bounded in $W$, otherwise $N_{\dom f_{|_W}}(0)$ contains a line, which is not possible due to $ \ri_{\aff(L\cap\dom g)}(\dom g_{|_W}) \neq \emptyset$ (it can be seen using similar arguments as those given in the proof of \cref{corollaryconvex1}). Then, we can assume that there exists 
 	   $F \in \Pf(T)$, $x \in  W $, $x^\ast_t \in \sub (g_t)_{|_W}(x)$ and  $(\lambda_t) \in \Delta(F)$ such that $\max_{t\in F}|\langle x^\ast_t, x \rangle| \leq  \epsilon/5 $, $|g_t(x) - g(0)| \leq  \epsilon/5 $, $g_t(x)=f_{F}(x)$ for all $t\in F$ and 
 	    $$x^*_{|_W} \in \sum_{t\in F} \lambda_t x_t^\ast + (P^*)^{-1}(U),$$
 	     where $P$ is a continuous projection from $X$ to $W$. Then,
 	     	\begin{align}
 	     	x^\ast\in  \sum_{t\in F} \lambda_t w_t^* + x^\ast - P^*(y^*_{|_W} )  +U,\label{Weee}
 	     \end{align}
      here $w_t^\ast :=P^{*}(x_t^\ast)$ and $w^\ast_t \in \sub ( g_t + \delta_{W} )(x)$. Furthermore, for all $t\in F$ 
 	     \begin{align}\label{Weee34}
 	     	\begin{array}{rl}
 	     		 f_t(0)+  2\epsilon/5  &\geq  g_t(0)+  2\epsilon/5 \geq  g_t(0) + |\langle x_t ,x\rangle|+  \epsilon/5 \\
 	     		&	\geq g_t(x) +  \epsilon/5  \geq g(0)=f(0),
 	     	\end{array}
 	     	\end{align}

 	  Now by Hirriat-Hurruty-Phelps' formula \cite[Theorem 2.1]{MR1245600}
 	 	$$\sub ( g_t + \delta_{W} )(x) \subseteq  \sub_{ \epsilon/5 } g_t(x) +W^{\perp} + U,$$
 	  which implies the existence of some point $\tilde{w}^\ast_t \in    \sub_{ \epsilon/5 } g_t(x) $ such that 
 	  \begin{align}\label{Weeee}
 	w^\ast_t  \in \tilde{w}^\ast_t  +  W^{\perp} + U.
 	  \end{align}
 	   Now, let us show that $ \tilde{w}^\ast_t  \in  \sub_{\epsilon} f_t(0)$. Indeed, consider $z\in X$, then
 	   	  \begin{align*}
 	  	\langle \tilde{w}^\ast_{t}, z  \rangle&= \langle \tilde{w}^\ast_{i,t}, z - x \rangle  + | \langle w^\ast_{t}, x \rangle|\leq g_t(z)- g_t(x) +  \epsilon/5 +  \epsilon/5  \\
 	  	&\leq g_t(z)-g_t(0) + g_t(0) -g_t(x) + 2 \epsilon/5  \leq g_t(z)-g_t(0) + g(0) -g_t(x)  + 2 \epsilon/5 \\
 	  	& \leq  g_t(z)-g_t(0) + 3\epsilon/5\leq f_t(x) - f_t(0) + f_t(0)- g_t(0) + 3\epsilon/5\\
 	  	&\leq f_t(z) - f_t(0) + \epsilon \text{ (by \cref{Weee34})}.
 	  \end{align*}
  Now, according to \cref{Weee,Weee34,Weeee}   we get \cref{primero} and from the arbitrariness of $\epsilon>0$ and $U$ we conclude that
   \begin{align}\label{segundo}
   	S_L(0) + N_{D \cap L}(0)  \subset	\bigcap\limits_{\substack{\epsilon>0}}  \cl^{w^*}\bigg(  \co \big(  \bigcup\limits_{ t  \in  T_\epsilon(x) }  \sub_{\epsilon}f_t(x) \big) + N_{D \cap L}(x)\bigg).
   	\end{align}
   	
   	Finally, using     \cref{zero} and \cref{segundo}  we conclude the desired inclusion in \cref{FORMULA:CONVEX2}.
\end{proof}

\begin{remark}
	It is worth mentioning that \cref{teofinal} represents a slight extension of \cite[Corollary 6]{MR3561780}, because in this result the authors have assumed that the data functions $f_t$'s are convex and proper.  
\end{remark}

\section{Conclusions}

In this paper, we have provided general formulae for the supremum function of an arbitrary family of lsc functions.

In \Cref{sectionfuzzy}, we  provided general fuzzy calculus rules in terms of the Fr\'echet subdifferential. Our approach follows from establishing these fuzzy calculus rules for an increasing family of functions (see \cref{teo:sup:1,}), where the key tool is the introduction of the notation of \emph{robust infimum}.  Later, in \cref{THEOREM:FORMULA:SUPREMUM}, we used the power set ordered by inclusion to get general fuzzy calculus rules of an arbitrary family of functions, without any qualification condition, as far as we know this approach is novel. 

In \Cref{limitingSub} we  established the main results of the paper, where we  replaced the  Lipschitz continuous assumption of the data by some limiting condition in  terms of the singular subdifferentials  (see \cref{Teo2c,LIMITING:CONDITION}). It has not escape our notice that these kind of conditions are becoming more popular in providing subdifferential calculus rules (see, e.g., \cite{MR2191744,MR2191745,MR1491362,MR2986672,MR2144010,MR1014198,MR859504,MR719677}). This section was divided into \Cref{limitingSubfinite} and \Cref{limitingSubinfinite}, which focused  attention on finite-dimensional and  infinite-dimensional settings respectively.  In both subsections we gave formulae for the subdifferential of the supremum function under different conditions. Here, It is worth comparing \cref{teoremcompactindex} and  \cref{Mordukhovich:Separable}. The main difference between these two results is that the first one is a convex upper-estimate, and  the second one corresponds to a non-convex upper-estimate  (as we showed in \cref{Example:Nonconvex}). This difference can be explained, because \cref{teoremcompactindex}  uses  a limiting condition only at the point of interest (see, \cref{Teo2c}), but \cref{Mordukhovich:Separable} uses the information of the subdifferential at a  neighborhood of the point of interest (see \cref{LIMITING:CONDITION}).

Finally, in \Cref{SECTION:CONVEXSUB} we  shown that our approach can be used to get new  formulae for the convex subdifferential, with and without qualification conditions, of the supremum function (see \cref{TEO:CONVEX:1} and  \cref{corollaryconvex1}), and also, it allows us to recover  \cite[Corollary 6]{MR3561780} using \cref{Mordukhovich:Separable} (see \cref{teofinal}), which in particular shows a unifying approach to the study of the subdifferential of the supremum function.

\bibliographystyle{plain}
\bibliography{references}
\end{document}